\documentclass{amsart}
\pdfoutput=1

\usepackage{amsmath}
\usepackage{amssymb}
\usepackage{amsthm}
\usepackage{mathtools}
\usepackage{verbatim}
\usepackage{enumerate}

\usepackage{tikz}
\usepackage{subcaption}
\usetikzlibrary{hobby}
\usetikzlibrary{fadings}
\usetikzlibrary{positioning}
\tikzset{
    base/.style = {rectangle, rounded corners, draw=black, minimum width=4cm, minimum height=1cm, text centered, font=\sffamily},
    box/.style = {rectangle, rounded corners, minimum width=5cm, minimum height=2cm, text centered, draw=gray, fill=black!15},
    header/.style={
    label={[rectangle, fill=white, draw, anchor=center, minimum width=2cm, node font=\ttfamily, name=\tikzlastnode-header]north:{#1}}}}
\tikzstyle{arrow} = [very thick, ->, >=stealth]

\usepackage{shortcuts}

\begin{document}

\title{On Convex Domains maximizing the\\ gradient of the torsion function}

\begin{abstract} We consider the solution of $-\Delta u = 1$ on convex domains $\Omega \subset \mathbb{R}^2$ subject to Dirichlet boundary conditions $u =0$ on $\partial \Omega$. Our main concern is the behavior of $\|\nabla u\|_{L^{\infty}}$, also known as the maximum shear stress in Elasticity Theory and first investigated by Saint Venant in 1856. We consider the two shape optimization problems $\| \nabla u\|_{L^{\infty}}/ |\Omega|^{1/2}$ and $\| \nabla u\|_{L^{\infty}}/ \mathcal{H}^1( \partial \Omega)$. Numerically, the extremal domain for each functional looks a bit like the rounded letter `D'. We prove that (1) either the extremal domain does not have a $C^{2 + \varepsilon}$ boundary or (2) there exists an infinite set of points on $\partial \Omega$ where the curvature vanishes. Either scenario seems curious and is rarely encountered for such problems. The techniques are based on finding a representation of the functional using only conformal geometry and classic perturbation arguments.
\end{abstract}

\author{Linhang Huang}
\address{Department of Mathematics, University of Washington, Seattle, WA 98203}
\email{lhhuang@uw.edu}
\pagestyle{plain}

\maketitle

\section{Introduction and Results}

\subsection{Introduction}
We study a variational problem involving the torsion function. For any bounded, convex domain $\Omega \subset \mathbb{R}^2$, the \textit{torsion function} $u$ on $\Omega$ is the unique solution to the Dirichlet problem:
\begin{align*}
    -\Delta u = 1 \quad &\text{inside $\Omega$},\\
    u = 0 \quad &\text{on $\partial\Omega$}.
\end{align*} 
This is a classical object in Elasticity Theory and was already studied by Saint Venant \cite{saint} in 1856. From the very beginning, there has been substantial interest in trying to understand where $|\nabla u|$ is the largest; these points are the one's subjected to the most stress. Saint Venant called them `the dangerous points' (`\textit{les points dangereux}' \cite{saint}). These were also of interest to Lord Kelvin and Tait (see their 1867 \textit{Treatise on Natural Philosophy} \cite{kelvin}). Additional early work was undertaken by Boussinesq \cite{bouss}, Filon \cite{filon} and Griffith-Taylor \cite{griff}.  P\'olya \cite{pol1} showed in 1930 that the point maximizing $|\nabla u|$ is always on the boundary. The partial differential equation $-\Delta u = 1$, being particularly simple and nonetheless exhibiting very nontrivial behavior, has been studied for a very long time in a wide variety of different settings. Much is known about the partial differential equation $-\Delta u = 1$, we refer to \cite{bandle, ban3, ban4, beck, hermitemany, brascamp, dragomirk, kawohl, keady, stv1, henrot, li1, li2, jianfeng, jianfeng2, makai, makar, pay0, pay1, payshape, payque, payphi, payphi2, paywhe, phi, stv2, sperb, sperbpayne, sperb2, stv3, stv4} and references therein.

 \subsection{The Problem} We are concerned with a very simple question: how large can the gradient be, what are the best bounds for $\|\nabla u\|_{L^{\infty}(\Omega)}$?  It is relatively easy to see that the gradient need not be bounded in general domains; however, the expression is always finite in convex domains. There is a scaling symmetry, if we rescale the domain $\Omega \rightarrow \lambda \Omega$ then the gradient rescales by a factor of $\sqrt{\lambda}$. Accounting for this symmetry leads to two natural shape optimization problems
$$ \sup_{\text{\tiny $\Omega$~convex, bounded}} \frac{\Vert\nabla u\Vert_{L^\infty(\Omega)}}{|\Omega|^{1/2}} \qquad \mbox{and} \qquad  \sup_{\text{\tiny $\Omega$~convex, bounded}}  \frac{\Vert\nabla u\Vert_{L^\infty(\Omega)}}{\mathcal{H}^1(\partial \Omega)}.
$$
The first problem has been studied for a very long time. A fundamental insight by Sperb \cite{sperb} is that the expression $|\nabla u|^2 + 2u$ attains its maximum at the boundary. Since $u$ vanishes on the boundary, we deduce the inequality
$$ \| \nabla u\|_{L^{\infty}(\Omega)}^2 =\left\Vert\frac{\partial u}{\partial\mathbf{n}} \right\Vert^2_{L^\infty(\partial\Omega)} \leq 2 \|u\|_{L^{\infty}(\Omega)}.$$
It is a celebrated result of P\'olya \cite{pol2} that $\|u\|_{L^{\infty}(\Omega)}$ is maximized when the domain is a ball and a direct computations gives $\|u\|_{L^{\infty}(\Omega)} \leq |\Omega|/(4\pi)$. Altogether,  
$$ \frac{\Vert\nabla u\Vert_{L^\infty(\Omega)}}{|\Omega|^{1/2}} \leq \frac{1}{\sqrt{2\pi}} \approx 0.398\dots$$
It is known that this upper bound is not tight. A simple computation using ellipses (for which $u$ is merely a quadratic polynomial) gives an example for which the functional is at least $0.321$. Two independent numerical approaches \cite{hoskins2021towards} both suggest the optimal value to be roughly $0.358$. The second functional, which was suggested by Krzysztof Burdzy, can be bounded in terms of the first: using the isoperimetric inequality and the previous bound, one has
$$ \frac{\Vert\nabla u\Vert_{L^\infty(\Omega)}}{\mathcal{H}^1(\partial \Omega)} \leq \frac{\Vert\nabla u\Vert_{L^\infty(\Omega)}}{|\Omega|^{1/2}} \frac{|\Omega|^{1/2}}{\mathcal{H}^1(\partial \Omega)} \leq \frac{1}{\sqrt{2\pi}} \frac{1}{2 \sqrt{\pi}} = \frac{1}{\sqrt{8} \pi}.$$

Explicit examples, see Figure~\ref{fig:optimal}, shows that both inequalities are nearly tight, the first being as little as $10\%$ away from the optimal value, the second at most $13\%$.

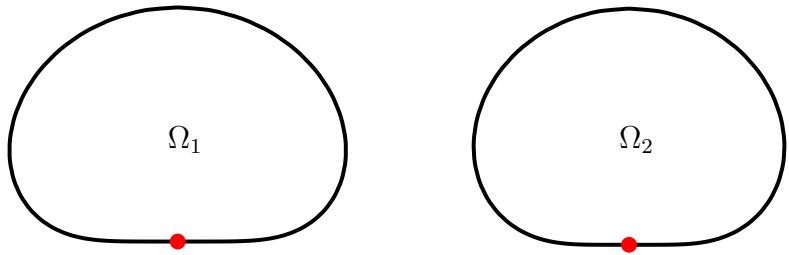
\begin{figure}[ht]
    \centering
    \begin{tikzpicture}\draw[domain=0:360,smooth,samples=50,variable=\t, line width=1.5]
  plot(
  {-3+1.5*( 
  -0.102255193272555*sin(\t)^3 + 0.159549913816006*sin(\t)*cos(\t)^2 + 0.738939378235281*sin(\t)*cos(\t) + 1.38661223308667*sin(\t))}, {0.1+1.5*(-0.221233862285197*sin(\t)^2*cos(\t) - 0.512311190685248*sin(\t)^2 + 0.0383548502864039*cos(\t)^3 + 0.266454947029555*cos(\t)^2 + cos(\t))});
\draw[domain=0:360,smooth,samples=50,variable=\t, line width=1.5]
  plot({3+1.5*(
-0.0877652088366168*sin(\t)^3 + 0.17464899865143*sin(\t)*cos(\t)^2 + 0.769299201814156*sin(\t)*cos(\t) + 1.22783150469083*sin(\t))}, {1.5*(-0.214439542806933*sin(\t)^2*cos(\t) - 0.472284898260466*sin(\t)^2 + 0.0474139429240895*cos(\t)^3 + 0.313275558932602*cos(\t)^2 + cos(\t))}
  );

    \fill [red] ({-3+1.5*( 
  -0.102255193272555*sin(180)^3 + 0.159549913816006*sin(180)*cos(180)^2 + 0.738939378235281*sin(180)*cos(180) + 1.38661223308667*sin(180))}, {0.1+1.5*(-0.221233862285197*sin(180)^2*cos(180) - 0.512311190685248*sin(180)^2 + 0.0383548502864039*cos(180)^3 + 0.266454947029555*cos(180)^2 + cos(180))}
  ) circle (3pt);
  
    \fill [red] ( {3+1.5*(
-0.0877652088366168*sin(180)^3 + 0.17464899865143*sin(180)*cos(180)^2 + 0.769299201814156*sin(180)*cos(180) + 1.22783150469083*sin(180))}, {1.5*(-0.214439542806933*sin(180)^2*cos(180) - 0.472284898260466*sin(180)^2 + 0.0474139429240895*cos(180)^3 + 0.313275558932602*cos(180)^2 + cos(180))}
 ) circle (3pt);

    \node at (-2.9, 0.3) {\Large $\Omega_1$};
    \node at (3.1, 0.3) {\Large $\Omega_2$};
    \end{tikzpicture}
    \caption{Left: an example of a convex domain for which $\|\nabla u\|_{L^{\infty}}/|\Omega|^{1/2} \sim 0.8926/\sqrt{2\pi}$.  Right: am example of a domain where one has   $\|\nabla u\|_{L^{\infty}}/\mathcal{H}^1(\Omega) \sim 0.8723/(\sqrt{8}\pi)$. Both are from a family of domains described by Sweers (see \cite{hoskins2021towards} and Section 6).}
    \label{fig:optimal}
\end{figure}

We are motivated by the extremal domains for these two functionals. The shape shown in Figure~\ref{fig:optimal} for the first functional was deduced in two different ways, once via explicit numerics and once via conformal geometry \cite{hoskins2021towards}, and the results appear to be consistent and very close to what is shown in Figure 1. Since we are maximizing an $L^{\infty}-$norm, the asymmetry is perhaps not too surprising. The extremal domains appear to be smooth (though hypothetical non-smooth extremizers would be difficult to find numerically) and they appear to be relatively flat close to the point where $|\nabla u|$ attains its maximum. Virtually nothing seems to be known about these domains. Standard results could be used to show that the inradius cannot be too small or that the curvature cannot be uniformly very large, however, the existing techniques do not seem to easily allow for the study of these domains.

\subsection{Statement of Results}
We prove two structure statements about the possible shape of these extremizers. The first result states that the maximizer of $\|\nabla u\|_{L^{\infty}}/|\Omega|^{1/2}$ cannot have `sharp corners' (meaning corners with an angle of more than 90 degrees). 

\begin{theorem}[No Sharp Corners]\label{theo-1}
If $\Omega$ is a convex maximizer of $\|\nabla u\|_{L^{\infty}}/|\Omega|^{1/2}$ and if the boundary is $\partial \Omega$ is $C^{1+\alpha}$ around any maximizing point for some $\alpha>0$, then $\partial \Omega$ cannot admit two adjacent $C^{1+\beta}$-smooth subarcs that meet at an angle sharper than $\pi/2$ for any $\beta\in(0,1)$.
\end{theorem}

This rules out extremal examples like the square, however, one would naturally be inclined to believe that this limitation is an artifact of the method; in fact, the result follows from showing that the Riemann map $f:\mathbb{D} \rightarrow \Omega$ has to satisfy $f' \in L^2(\partial\mathbb{D})$ which rules out corners sharper than $90$ degrees. 
The second statement is our main result. It shows that the extremizer of either of these two problems is guaranteed to be fairly unusual (and thus fairly interesting).
\begin{theorem}[Main Result]\label{theo}
    Let $\Omega$ be a convex maximizer of either of the two functionals. Then
    \begin{enumerate}
        \item either the boundary is not $C^{2+\varepsilon}$ for some $\varepsilon >0$ or
        \item or the points of zero curvature on $\partial \Omega$ accumulate around the point in which the gradient is maximal.
    \end{enumerate}
\end{theorem}

Both cases would be \textit{very} interesting.  Given the nature of the problem, one is naturally inclined to believe that the extremal domain is going to be smooth. There are natural questions of the precise type of regularity (real-analytic, Gevrey, $C^{\infty}$, ...) but there is nothing in the problem statement that would suggest the possibility of $\partial \Omega$ being not, say, $C^3$. The second alternative would mean that $\partial \Omega$ is `flat' around the point in which $|\nabla u|$ is maximal. This is natural and to be expected, however, the natural expectation is surely that the curvature $\kappa$ vanishes, say, quadratically, or to some higher order. Having a sequence of points with zero curvature in every neighborhood is very atypical.

\subsection{The techniques}
Our techniques rely entirely on classic arguments from complex analysis. This is partially inspired by the hope that one might one day be able to characterize the extremal domains. Given their unusual shape, the only hope of possibly achieving this might be by describing it in terms of the Riemann map. Instead of optimizing over convex domains $\Omega$, we optimize over conformal maps $f:\mathbb{D} \rightarrow \Omega$. For example, if $f(\mathbb{D}) = \Omega$ has a $C^{1+\alpha}$-boundary and if $|\nabla u|$ is maximal in $f(1) \in \partial \Omega$, then \begin{equation*}
    \frac{\Vert\nabla u\Vert_{L^\infty(\Omega)}}{|\Omega|^{1/2}} = \frac{\int_\D |f'(w)|^2 P(w)|dw|^2}{2\pi|f'(1)| \cdot \left( \int_\D |f'(w)|^2|dw|^2\right)^{1/2}}, \qquad (\diamond).
\end{equation*}
where $P:\mathbb{D} \rightarrow \mathbb{C}$ is an explicit function. We prove a precise equivalence between the PDE problem and its analogue in conformal geometry.
One persistent obstacle is that convexity of $\Omega$ is non-trivially encoded in the behavior of the conformal map $f:\mathbb{D} \rightarrow \Omega$ (Study's Theorem etc). The main idea is then to use an Euler-Lagrange approach within a suitable function space and with a very precisely constructed deformation. It would be nice if the extremal domain for either of the two problems had a nice representation in terms of the conformal map or, equivalently, if the functional $(\diamond)$ had a nice extremizers; once more, the global restriction of $f(\mathbb{D})$ being convex poses a significant challenge.

\section{Preliminaries}\label{prelim}
\subsection{Conformal Geometry on Convex Domains}
In this paper, we will follow \cite{beliaev2020conformal, duren2001univalent, garnett2005harmonic} and use the term `conformal map' to refer to analytic univalent maps. We will use the complex variable $w$ for area integrals and $\zeta$ for line integrals. That is, we will write $|dw|^2$ for the Lebesgue measure on $\R^2=\C$ and $|d\zeta|$ for the Hausdorff measure $\mathcal{H}^1(\Gamma)$ on a rectifiable curve (or loop) $\Gamma\subseteq \C$.

Given a Jordan domain $\Omega\subseteq \C$, as it is simply connected, there exists a conformal map $f: \D \rightarrow \Omega$. We refer to this map as a Riemann map of $\Omega$. By Carathéodory's theorem \cite{caratheodory1913begrenzung}, $f$ extends uniquely to a homeomorphism between $\overline{\D}$ and $\overline{\Omega}$. Convexity of a domain in $\mathbb{R}^2$ is a purely geometric condition: there are a number of classical results explaining how this notion is reflected in the Riemann map and we summarize three such results. 
\begin{lemma}[{Characterization of Convexity, see \cite[Duren, Section 2.5]{duren2001univalent}}]\label{harmonic}
    Let $f:\D\to \C$ be a locally univalent analytic function (that is, $f'(z)\neq 0$). Then $f$ is a conformal map onto a convex domain if and only if \begin{equation*}
        \re\left(1+z\frac{f''(z)}{f'(z)}\right) >0\quad \text{for all}\quad z\in \D.
    \end{equation*}
\end{lemma}
There is another characterization of convex conformal maps by concentric disks.
\begin{lemma}[{Study's Theorem,  \cite[Section 13]{emch1912study}}]\label{convex}
    Let $f:\D\to\Omega$ be a locally univalent function. Then $f$ is a conformal map onto convex domain if and only if for any $r\in (0,1)$, $f(r\D)$ is a convex domain.
\end{lemma}
\begin{proof}
Lemma \ref{harmonic} allows for a very short proof: let $f_r(z):=f(rz)$. If $f$ is a conformal map of a convex domain, then we have $f'_r(z) = rf'(rz)\neq 0$ and \begin{equation*}
        \re\left(1+z\frac{f''_r(z)}{f'_r(z)}\right) = \re\left(1+z\frac{r^2f''(rz)}{rf'(rz)}\right)= \re\left(1+rz\frac{f''(rz)}{f'(rz)}\right)>0.
    \end{equation*} 
\end{proof}
The third characterization of convex conformal maps is stated as follows. 
\begin{lemma}[{Radial Increase,  \cite[Proposition 2.13]{pascu2002scaling}}]\label{increase}
    Let $f: \D\to \Omega$ be a conformal map onto a simply connected domain $\Omega$. Then $\Omega$ is convex if and only if \begin{equation*}
        r |f'(re^{i\theta})| \quad \text{is an increasing function of $r\in [0,1)$  for all $\theta \in [-\pi, \pi]$}.
    \end{equation*}
\end{lemma}
\begin{proof}
    For the convenience of the reader, we sketch a quick proof using only Lemma \ref{harmonic}. Given $\theta\in[-\pi,\pi]$, we let $h_\theta(r):= r^2|f'(re^{i\theta})|^2$. Then we have \begin{align*}
        h'_\theta(r) &= 2r|f'(re^{i\theta})|^2 + r^2\frac{d}{dr}f'(re^{i\theta})\overline{f'(e^{i\theta})}\\
         &= 2r|f'(re^{i\theta})|^2 + r^2\left(e^{i\theta}f''(re^{i\theta})\right)\overline{f'(e^{i\theta})}+ r^2f'(e^{i\theta})\overline{\left(e^{i\theta}f''(re^{i\theta})\right)}\\
         &= 2r|f'(re^{i\theta})|^2 +2r^2\re\left(e^{i\theta} f''(re^{i\theta})\overline{f'(re^{i\theta})}\right)\\
         &= 2r|f'(re^{i\theta})|^2\re\left( 1+\frac{re^{i\theta} f''(re^{i\theta})}{f'(re^{i\theta})}\right).
    \end{align*} The result then follows from Lemma~\ref{harmonic}.
\end{proof}

\subsection{Two Types of Perturbations}
Given a Riemann map $f$ of a bounded convex domain $\Omega$, in this paper, we will introduce two ways to "perturb" the map $f$ to generate new convex domains. First, we upgrade Lemma~\ref{convex}. 

\begin{corollary}\label{circles}
    Let $f$ be a Riemann map of bounded convex domain. Then for any disk $B\subset \D$, $f(B)$ is also a convex domain.
\end{corollary}
\begin{proof}
    We divide the proof into two cases: the case   $\overline{B}\subseteq\D$ and the case where  $\partial B$ touches $\partial \D$.  We start with the case $\overline{B}\subseteq\D$. Here, we will show that we can map $B$ to some $r\D$ using a Möbius transformation of the unit disk. Consider the Möbius transformation $\varphi: \D\to\H$ given by \begin{equation*}
        \phi(z)=  i\frac{1+z}{1-z}.
    \end{equation*}
    This map (see Figure~\ref{fig:concirc}) satisfies that \begin{equation*}
        \varphi(r\D) = B_\frac{2r}{1-r^2}\left(i\frac{1+r^2}{1-r^2}\right).
    \end{equation*} Setting $\lambda = (1+r^2)/(1-r^2)$, we have $
        \varphi(r\D) = B_{\sqrt{\lambda ^2-1}}(\lambda i)$,
      where $r\in(0,1)$ corresponds to $\lambda\in(1,\infty)$ bijectively (we write $r = r(\lambda)$). 
      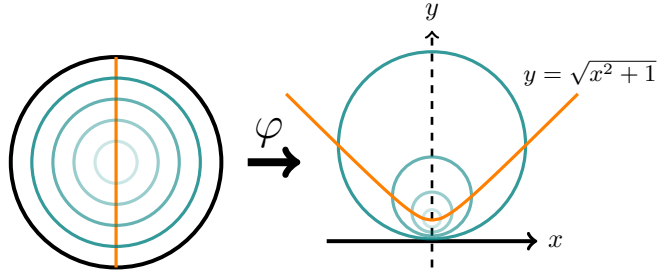
\begin{figure}[ht]
        \begin{tikzpicture}[scale=0.7]
            \draw [line width=1.5] (-3, 0) circle (2);
            \draw [line width=1.2, color=orange](-3,-2)--(-3,2);
            \draw[->, line width=1.5] (1,-1.5)--(5,-1.5) node[right]{$x$};
            \draw[->, dashed, line width=1] (3,-2)--(3,2.5) node[above]{$y$};
            \tc{0.2}
            \tc{0.4}
            \tc{0.6}
            \tc{0.8}
            \draw[domain=1:7, samples=100, line width=1.2, color = orange, smooth, variable=\t] plot ({0.4*sqrt(\t^2-1)+3},{0.4*\t-1.5});
                \draw[domain=1:7, samples=100, line width=1.2, color = orange, smooth, variable=\t] plot ({-0.4*sqrt(\t^2-1)+3},{0.4*\t-1.5});       
            \draw [->,line width=3pt] (-0.5,0) -- (0.5,0) node[pos=0.4,above=1pt] {\huge $\varphi$};
            \node at (6,1.7) {\small $y=\sqrt{x^2+1}$}; 
        \end{tikzpicture}
        \caption{Image of concentric circles under $\varphi$}
        \label{fig:concirc}
    \end{figure}
    Suppose $\varphi$ maps $B$ to a disk $B_c(a+bi)$ in $\H$. Since $\overline{B}\subseteq \D$, we have that $c<b$. Note that \begin{equation*}
        \varphi_0(B_c(a+bi)) = B_{\frac{c}{b^2-c^2}}\left(\frac{bi}{b^2-c^2}\right) = B_{\sqrt{\lambda'^2-1}}(\lambda'i) = \varphi(r'\D),
    \end{equation*} where $\lambda'=b/(b^2-c^2)$, $r' = r(\lambda')$ and \begin{equation*}
        \varphi_0(z) = \frac{z-a}{b^2-c^2}.
    \end{equation*}
    is an automorphism of $\H$. Then $\varphi_0(B_r(a+bi))$ is the image of some $r'\D$ under $\varphi$. Hence, $B_c(a+bi)$ is the image of some $r'\D$ under $\varphi_0^{-1}\circ \varphi:\D\to\H$. This shows that $B$ is an image of some $r\D$ under the Möbius transform $\varphi_B=\varphi^{-1}\circ\varphi_0\circ \varphi\in \mathrm{Aut}(\D)$. Note that $f\circ \varphi_B^{-1}$ is a conformal map with convex image. Thus $f\circ \varphi_B^{-1}(r\D) = f(B)$ is an analytic convex domain by Lemma~\ref{convex}.  
    It remains to deal with the case where $\partial B$ touches $\partial \D$: without the loss of generality, we suppose that $\partial B\cap\partial \D = \set{1}$. Let $t>0$ be such that $\overline{B}-t \subseteq \D$. Then we have that $f(B)$ is a topological limit of $f(B-t)$ as $t\to 0$. The result follows from the first case.
\end{proof}

By restricting $f$ to a smaller disk $B\subseteq\D$, we can get a convex sub-domain of the original convex domain and its corresponding Riemann map. We refer to this method for producing new convex maps as \textit{domain restriction} (see Figure~\ref{fig:dr}).

\begin{figure}[ht]
    \centering
    \begin{tikzpicture}\draw [line width=1.5] (-6, 0) circle (2);
    \fill[color=orange, opacity=0.7] (-5.86, 0.1) circle (1.66);
    \draw[domain=0:360,smooth,samples=30,variable=\t, line width=1.5, draw opacity=1]
  plot({2*(sin(\t)-0.1*sin(3*\t))}, {2*(cos(\t)-0.1*cos(3*\t))});
\fill[domain=0:360,smooth,samples=30,variable=\t, color=orange, opacity=0.7]
  plot({2*(0.83*sin(\t) + 0.05 - 0.1*(-0.83*sin(\t) - 0.05)^3 - 0.3*(0.83*sin(\t) + 0.05)*(0.83*cos(\t) + 0.07)^2)}, {2*(0.83*cos(\t) + 0.07 + 0.3*(0.83*sin(\t) + 0.05)^2*(0.83*cos(\t) + 0.07) - 0.1*(0.83*cos(\t) + 0.07)^3)});
    \draw [->,line width=3pt] (-3.5,0) -- (-2.5,0) node[pos=0.4,above=1pt] {\huge $f$};
    \node at (-5.85, 0.1) {\Huge $B$};
    \node at (0.2, 0.1) {\Huge $f(B)$};
    \end{tikzpicture}
    \caption{Domain restriction.}
    \label{fig:dr}
\end{figure}
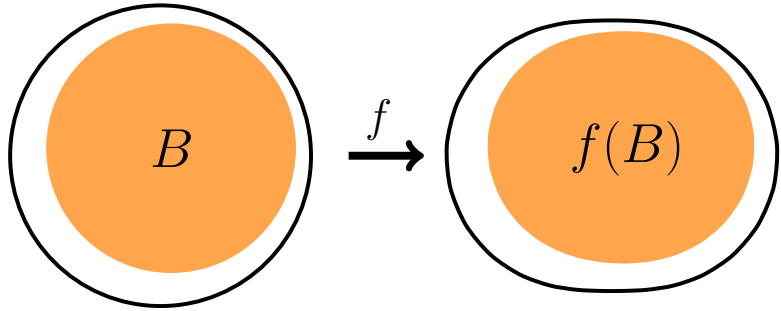

Secondly, we consider adding a noise to the map $f$ in the direction of a new function $g:\mathbb{D} \rightarrow \mathbb{C}$. We set \begin{equation*}
    v_f(z):=\re\left(1+z\frac{f''(z)}{f'(z)}\right).
\end{equation*}
Note that $v_f(z)$ is a harmonic function. We can set \begin{equation*}
    k_f(\theta) := \liminf_{\D \ni z\to e^{i\theta}}v_f(z).
\end{equation*}

We call $k_f:[-\pi,\pi]\to [-\infty,\infty]$ the \textit{curvature signature} of $f$. We will discuss its relation to the curvature in Section~\ref{perturbation}. Note that \begin{equation*}
    v_f(0)=\re\left(1+0\cdot\frac{f''(0)}{f'(0)}\right)=1.
\end{equation*} It follows that if $k_f$ is a continuous extention of $v_f$ to $\partial\D$, then we have \begin{equation*}
    \int_{-\pi}^\pi k_f(\theta)d\theta= 2\pi v_f(0) = 2\pi.
\end{equation*} 

\begin{corollary}[Characterization by the curvature signature]\label{con_sign}
    Let $f:\D\to \C$ be a locally univalent analytic function. Then $f$ is a conformal map onto a convex domain if and only if $k_f(\theta) \geq 0$.
\end{corollary}

\begin{proof}
    It is clear that $k_f\geq 0$ if $v_f >0$. Suppose that $k_f$ is nonnegative. Then $v_f$ is lower-bounded. Thus by the maximum (minimum) principle, we have \begin{equation*}
        v_f(z) \geq \inf_{\theta\in[-\pi,\pi]} k_f(\theta) \geq 0.
    \end{equation*}
    Since $v_f$ is a non-negative harmonic function that is strictly positive at $1$, it is strictly positive on $\D$.
\end{proof}

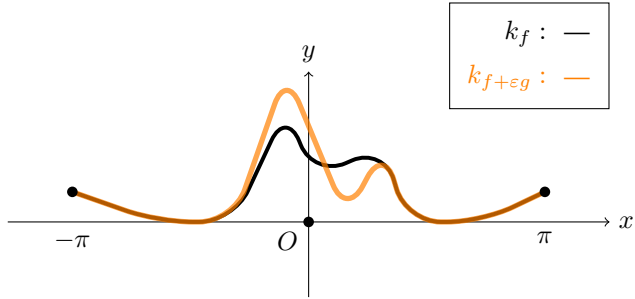
\begin{figure}[ht]
    \centering
    \begin{tikzpicture}
        \draw[->] (-4,0)--(4,0) node[right]{$x$};
        \draw[->] (0,-1)--(0,2) node[above]{$y$};

        \draw[line width = 1.5, rounded corners = 13] (3.14, 0.4) -- (2.3, 0) -- (1.3,0) -- (1, 1) --(0.07,0.6) -- (-0.3,1.5) -- (-1,0) -- (-2,0) -- (-3.14,0.4);
        \draw[line width = 2, rounded corners = 13, color=orange, opacity=0.7] (3.14, 0.4) -- (2.3, 0) -- (1.3,0) -- (1, 1) -- (0.5,0.07) -- (-0.3,2) -- (-1,0) -- (-2,0) -- (-3.14,0.4);

        \fill [black] (0,0) circle (2pt) node [black, below left=0.6pt]{$O$};
        \fill [black] (-3.141,0.4) circle (2pt) node [black, below=0.4]{$-\pi$};
        \fill [black] (3.141,0.4) circle (2pt) node [black, below =0.4]{$\pi$};

        \matrix [draw, above left] at (4,1.5) {  \node[black] {$k_f:$};& \node[color=black, font=\huge] {\textbf{\---}};\\  \node[orange] {$k_{f+\varepsilon g}:$};&\node[color=orange, font=\huge] {\textbf{\---}};\\
};
    \end{tikzpicture}
    \caption{Changing the curvature signature.}
    \label{fig:cs}
\end{figure}

Corollary~\ref{con_sign} gives us the second way to perturb a conformal map $f$: When adding a noise function $g$ to the original conformal map $f$, $f+\varepsilon g$ is a conformal map onto a convex domain so long as it is locally univalent and the curvature signature $k_{f+\varepsilon g}$ remains non-negative (see Figure~\ref{fig:cs}). 

\subsection{Setup for Torsion Functions}\label{setup} 
Given a Jordan domain $\Omega$ and its torsion function $u$, we set \begin{equation}
    \mathbb{L}_1(\Omega) := \frac{\left\Vert\nabla u \right\Vert_{L^\infty(\Omega)}}{|\Omega|^{1/2}} \qquad \mbox{and} \qquad \mathbb{L}_2(\Omega) := \frac{\left\Vert\nabla u \right\Vert_{L^\infty(\Omega)}}{\mathcal{H}^1(\partial\Omega)}.
\end{equation} Recall that we are interested in maximizing $\mathbb{L}_i$ over all convex bounded domains. Note that we have \begin{equation*}
    \| \nabla u\|_{L^{\infty}(\Omega)}^2 =\left\Vert\frac{\partial u}{\partial\mathbf{n}} \right\Vert^2_{L^\infty(\partial\Omega)}.
\end{equation*} We call the points on $\partial \Omega$ where $\left\Vert \partial u/\partial\mathbf{n} \right\Vert_{L^\infty(\partial\Omega)}$ are attained the \textit{extremal points}.
We will rewrite the quantities $\mathbb{L}_i(\Omega)$ in terms of a Riemann map $f$ of $\Omega$. Let us first recall Kellogg's theorem, which states the connection between the differentiability of $f$ on $\overline{\D}$ and the smoothness of $\partial \Omega$. 

\begin{lemma}[Kellogg's Theorem, \cite{kellogg1931derivatives}] Suppose $f:\D\to\Omega$ is a conformal map onto a Jordan domain $\Omega$. Let $k \geq 1$ and $0 < \alpha < 1$. Then the following are equivalent:
\begin{enumerate}[(a)]
    \item $\partial\Omega$ is a $C^{k+\alpha}$-smooth curve.
    \item $\arg f' \in C^{k-1+\alpha}(\partial \D)$.
    \item $f\in C^{k+\alpha}(\overline{\D})$ and $f'\neq 0$ on $\partial\D$.
\end{enumerate}
If $k\geq 1$ and $\alpha = 0$, then (a) and (b) are still equivalent and (c) implies both. However, (a) or (b) does not imply (c).     
\end{lemma}

With Kellogg's theorem, we can find a representation of the maximal normal derivative purely in terms of the Riemann map for convex $C^{1+\alpha}$ domains.

\begin{lemma}\label{function}
    Let $\Omega$ be a bounded convex $C^{1+\alpha}$ domain for some $\alpha>0$. Then for any Riemann map $f\in C^{1+\alpha}(\overline{\D})$ of $\Omega$ that sends $1$ to an extremal point on $\partial \Omega$,  
    \begin{equation*}
        \left\Vert\frac{\partial u}{\partial \mathbf{n}}\right\Vert_{L^{\infty}(\partial\Omega)} = \frac{1}{2\pi|f'(1)|}\int_\D |f'(w)|^2  \frac{1-|w|^2}{|1-w|^2} |dw|^2.
    \end{equation*}
\end{lemma}

\begin{proof}
    Suppose that $\Omega$ is $C^{1+\alpha}$ with torsion function $u$. Let $f:\D\to \Omega$ be the Riemann map of $\Omega$. By Kellogg's theorem, we have $f\in C^{1+\alpha}(\overline{\D})$ and thus $f'$ is $\alpha$-Hölder continuos on $\overline{\D}$. Given the function $v = u\circ f$, we have \begin{align*}
        \frac{\partial^2v}{\partial x^2} (x,y) &= \frac{\partial^2u}{\partial x^2} (f(x,y)) \cdot \left(\frac{\partial \re f}{\partial x}(x,y)\right)^2 + \frac{\partial^2 u}{\partial y^2} (f(x,y)) \cdot \left(\frac{\partial \im f}{\partial x}(x,y)\right)^2 \\
        &+ 2\frac{\partial^2 u}{\partial x\partial y} (f(x,y)) \cdot \frac{\partial \re f}{\partial x}(x,y)\cdot \frac{\partial \im f}{\partial x}(x,y)\\
        &+ \frac{\partial u}{\partial x } (f(x,y)) \cdot \frac{\partial^2 \re f}{\partial x^2}(x,y) +   \frac{\partial u}{\partial y} (f(x,y)) \cdot \frac{\partial^2 \im f}{\partial x^2}(x,y)
       \end{align*}
       as well as
       \begin{align*}
        \frac{\partial^2v}{\partial y^2} (x,y) &= \frac{\partial^2u}{\partial x^2} (f(x,y)) \cdot \left(\frac{\partial \re f}{\partial y}(x,y)\right)^2 + \frac{\partial^2 u}{\partial y^2} (f(x,y)) \cdot \left(\frac{\partial \im f}{\partial y}(x,y)\right)^2 \\
        &+ 2\frac{\partial^2 u}{\partial x\partial y} (f(x,y)) \cdot \frac{\partial \re f}{\partial y}(x,y)\cdot \frac{\partial \im f}{\partial y}(x,y)\\
        &+ \frac{\partial u}{\partial x } (f(x,y)) \cdot \frac{\partial^2 \re f}{\partial y^2}(x,y) +   \frac{\partial u}{\partial y } (f(x,y)) \cdot \frac{\partial^2 \re f}{\partial y^2}(x,y).
    \end{align*}
By the Cauchy-Riemann equations, we have \begin{align*}
    \frac{\partial \re f}{\partial x} = \frac{\partial \im f}{\partial y} \qquad \mbox{and} \qquad 
    \frac{\partial \re f}{\partial y} = -\frac{\partial \im f}{\partial x}.
\end{align*} 
By plugging in the equations, it follows that $v$ solves the Poisson equation.
\begin{align}\label{torsion}
    -\Delta v = |f'|^2 \quad &\text{inside $\D$},\\
    v = 0 \quad &\text{on $\partial\D$}.\nonumber
\end{align} The unique solution to $\eqref{torsion}$ is given in terms of the Green's function $G(z,w)$ on $\D$: \begin{equation*}
    v(z) =  \int_\D G(z,w)|f'(w)|^2|dw|^2,\\
\end{equation*} where the Green's function on $\D$ has the expression (see, for example, \cite{garnett2005harmonic}) 
\begin{equation*}
    G(z,w) = \frac{1}{2\pi} \log \frac{|z-w|}{|1-\bar w z|}.
\end{equation*} 

Since $|f'|^2$ is Hölder-continuous on the closed unit disk $\overline{\D}$, \cite[Lemma 4.1]{gilbarg1977elliptic} tells us that the normal derivative $\partial v/\partial \mathbf{n}$ exists in the strong sense on $\partial \D$ and is continuous. It has the representation \begin{equation*}
    \frac{\partial v}{\partial \mathbf{n}}(e^{i\theta})= -\int_\D P(e^{i\theta}, w) |f'(w)|^2|dw|^2,
\end{equation*} where \begin{equation*}
    P(\zeta,w) := \frac{\partial G(\zeta, w)}{\partial \mathbf{n}_\zeta} = \frac{1}{2\pi} \frac{1-|w|^2}{|\zeta-w|^2}
\end{equation*} is the Poisson kernel of $\D$.
Note that by composition, we have \begin{equation}\label{compo}
    \nabla v(z) = f'(z) \cdot \nabla u(f(z)),
\end{equation} where the gradients are treated as complex numbers. Since both $u$ and $v$ vanish on the $C^{1+\alpha}$ boundary, the regular level set theorem \cite{lee2012introduction} tells us that \begin{equation*}
    |\nabla v(e^{i\theta})| = -\frac{\partial v}{\partial \mathbf{n}}(e^{i\theta}),\quad |\nabla u(f(e^{i\theta}))| = -\frac{\partial u}{\partial \mathbf{n}}(f(e^{i\theta})).
\end{equation*} It follows that \begin{equation*}
    \frac{\partial v}{\partial \mathbf{n}}(e^{i\theta}) = |f'(e^{i\theta})|\frac{\partial u}{\partial \mathbf{n}}(f(e^{i\theta})).
\end{equation*}

This shows that $-\partial u/\partial \mathbf{n}$ is continuous (in fact it has to be as $\Delta u \in C^{\infty}(\overline{\Omega})$ and $u\in C^{\infty}(\partial \Omega)$). Thus the extremal points exist. Using the conformal map $f: \D\to \Omega$ mapping $1$ to one of the extremal points, we have \begin{equation*}
    \left\Vert\frac{\partial u}{\partial \mathbf{n}}\right\Vert_{L^{\infty}(\partial\Omega)} = -\frac{1}{|f'(1)|}\frac{\partial v}{\partial \mathbf{n}}(1).
\end{equation*} 
The result then follows.
\end{proof}

With some abuse of the notation, we also refer to \begin{align*}
    P(w) &= \frac{1-|w|^2}{|1-w|^2} = \frac{1}{1-\bar w} - \frac{1}{1-1/w} =\frac{1}{1-w}+\frac{1}{1-\bar w} -1
\end{align*} as the Poisson kernel. 
Note that for a $C^{1+\alpha}(\overline{\D})$ Riemann map $f:\D\to \Omega$ of a Jordan domain $\Omega$ with $\alpha \geq 0$, we have that \begin{align*}
    |\Omega| = \int_\D |f'(w)|^2|dw|^2,\quad \mathcal{H}^1(\partial \Omega) = \int_{\partial\D} |f'(\zeta)||d\zeta|.
\end{align*}

It follows that we can express $\mathbb{L}_1(\Omega)$ and $\mathbb{L}_2(\Omega)$ in terms of a Riemann map of $\Omega$.

\begin{proposition}\label{verysmooth}
    Let $\Omega$ be a bounded convex $C^{1+\alpha}$ domain for some $\alpha\in(0,1)$. For any Riemann map $f$ of $\Omega$ that sends $1$ to an extremal point on $\partial\Omega$, we have \begin{align*}
         L_1(f):&= \frac{\int_\D |f'(w)|^2 P(w)|dw|^2}{2\pi|f'(1)|\left(\int_\D |f'(w)|^2|dw|^2\right)^{1/2}} = \mathbb{L}_1(\Omega) 
         \end{align*}
         and
         \begin{align*}
        L_2(f):&= \frac{\int_\D |f'(w)|^2 P(w)|dw|^2}{2\pi|f'(1)| \int_{\partial \D} |f'(\zeta)||d\zeta|}=\mathbb{L}_2(\Omega).\\
    \end{align*}
\end{proposition}

\subsection{$L_i$ is well-defined.} It will be important for us to be able to use Proposition \ref{verysmooth} also in a slightly rougher setting. Indeed, we do not necessarily need the conformal map $f$ to be $C^{1+\alpha}(\overline{\D})$, the functional tells us what is needed: we require that \begin{enumerate}
    \item $|f'(z)|^2$ is integrable on $\D$.
    \item $|f'(z)|^2P(z)$ is integrable on $\D$.
    \item $f'(z)$ is defined at $1$. 
    \item $f'(z)$ is defined almost everywhere on $\partial\D$ and $|f'(z)|$ is integrable on $\partial\D$.
\end{enumerate}

If $f$ is a conformal map of a bounded domain, Condition~(1) is always true and we always have $\int_\D |f'(w)|^2|dw|^2 = |f(\D)|$. We will now show that in our setting condition (4) is also automatically satisfied. It then suffices to check Conditions (2) and (3). To check Conditions~(3) and (4), we need the concept of \textit{nontangential limits}. One can refer to \cite{garnett2005harmonic} for a more detailed introduction for nontangential limits. For $e^{i\theta} \in \partial \D$ and $\alpha > 1$, we define the \textit{cone} \begin{equation*}
    C_\alpha(\theta) = \set{
z \in \D: |z-e^{i\theta}| < \alpha(1-|z|)}
\end{equation*}
Note that $C_\alpha(\theta)$ is an increasing family of subsets in $\alpha$. The cone $C_\alpha(\theta)$ is asymptotic to a sector with vertex $e^{i\theta}$ and angle $2\sec^{-1}(\alpha)$ that is symmetric about the radius $[0, e^{i\theta}]$ (See Figure \ref{fig:cone}).

\begin{figure}[ht]
    \centering
    \begin{tikzpicture}\draw [line width=1.5] (0, 0) circle (2.5) ;
\fill[domain=0:360,smooth,samples=80,variable=\t, color=blue, opacity=0.5, line width=1.5]
  plot({-2.5*(-1.44+cos(60-\t)+sqrt((cos(60-\t)-1.44)^2-(1-1.44)^2))/(1.44-1)*cos(\t)},{-2.5*(-1.44+cos(60-\t)+sqrt((cos(60-\t)-1.44)^2-(1-1.44)^2))/(1.44-1)*sin(\t)});
    \draw node at (0.4,0.65) { $C_{1.2}(\frac{\pi}{3})$};
    \draw[line width=1.5, dashed] ({2.5*cos(60)},{2.5*sin(60)}) -- ({2.5*cos(60)+3.3*cos(0.578*360)},{2.5*sin(60)+3.3*sin(0.578*360)});
    \draw[line width=1.5, dashed] ({2.5*cos(60)},{2.5*sin(60)}) -- ({2.5*cos(60)+3.3*cos(0.755*360)},{2.5*sin(60)+3.3*sin(0.755*360)});
    \fill [teal] ({2.5*cos(60)},{2.5*sin(60)}) circle (3pt) node [black, above right=0.6pt]{$e^{i\pi/3}$};
     \draw[line width=1.5,<->] ({2.5*cos(60)+3*cos(0.582*360)},{2.5*sin(60)+3*sin(0.582*360)})  arc({0.582*360}:{0.75*360}:3) node[pos=.7, below left=0.3pt]{$2\sec^{-1}(1.2)$};
    \end{tikzpicture}
    \caption{Cone $C_{1.2}(\pi/3)$ highlighted in blue}
    \label{fig:cone}
\end{figure}
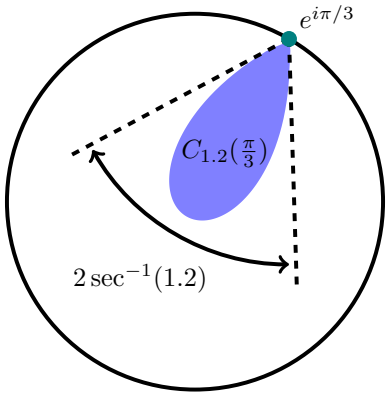

\begin{definition}[Nontangential limit]
    We say a function $u: \D \to \C$ has a \textit{nontangential limit} $A$ at $e^{i\theta}\in \partial \D$ if for all $\alpha >1$, we have \begin{equation*}
    \lim_{C_\alpha(\theta)\ni z \to e^{i\theta}} u(z) = A.
\end{equation*} 
\end{definition}
A celebrated result about nontangential limits of conformal maps is the following. \begin{lemma}[\cite{garnett2005harmonic}]\label{wfm}
    Let $\Omega$ be a Jordan domain and $f$ a conformal map from $\D$ onto $\Omega$. Then the curve $\partial \Omega$ is rectifiable if and only if $f'$ is in the Hardy space $\mathbf{H}^1$. If $f'\in \mathbf{H}^1$, then $f'$ admits has a nontangential limit $f'(e^{i\theta})$ almost everywhere on $\partial \D$ and \begin{equation*}
        \int_{\partial\D}|f'(\zeta)||d\zeta| = \mathcal{H}^1(\partial \Omega).
    \end{equation*}
\end{lemma}
Lemma~\ref{wfm} implies that if $f$ is a conformal map of a bounded convex domain (recall that closed convex curves are rectifiable), then Condition~4 is always satisfied and the integral $\int_{\partial\D} |f'(\zeta)||d\zeta|$ is always the length of $f(\partial\D)$.
\begin{definition}[Families of convex conformal maps]
    Let $\mathbf{F}$ be the set of all conformal maps of bounded convex domains. In particular, since the area is finite and given by $\int_{\mathbb{D}} |f'(z)|^2 dz$, we automatically know that all functions of interest satisfy $f' \in L^2(\D)$.
    That is, by Lemma~\ref{harmonic},\begin{equation*}
        \mathbf{F} := \set{f~\text{analytic on $\D$}~|~ f'\neq 0, v_f> 0 ~\text{on}~\D, \quad f' \in L^2(\D)},
    \end{equation*}
    Let $\mathbf{F}_0$ is a subfamily of $\mathbf{F}$ such that for each $f \in \mathbf{F}_0$, $|f'(z)|^2P(z)$ is integrable on $\D$ and $f'(z)$ has a nontangential limit at $1$. That is, \begin{equation*}
        \mathbf{F_0} := \set{f\in \mathbf{F}|~\text{$f'$ has a nontangential limit at $1$ and }|f'|^2P \in L^1(\D)}.
\end{equation*}
\end{definition}

The functional $L_1$ and $L_2$ are well-defined functionals on $\mathbf{F}_0$. We will show in Section~\ref{goingback} that \begin{align*}
    \sup_{f\in \mathbf{F}_0}L_1(f) &= \sup_{\text{\tiny $\Omega$ convex, bounded}}\mathbb{L}_1(\Omega),\\ \sup_{f\in\mathbf{F}_0}L_2(f) &= \sup_{\text{\tiny $\Omega$ convex, bounded}}\mathbb{L}_2(\Omega).
\end{align*}

For the ease of computation, we abbreviate 
$$
    \area(f) := \int_{\D}|f'(w)|^2|dw|^2
\qquad \mbox{and} \qquad     \len(f) := \int_{\partial\D}|f'(\zeta)||d\zeta|
$$
  whose names are a consequence of the fact that for sufficiently nice Riemann maps, $\area(f)$ is the area of the image and $\len(f)$ is the length of the boundary of the image. Motivated by the preceding considerations we also introduce the functional
  $$
    \poi(f) := \int_{\D}|f'(w)|^2P(w)|dw|^2.
$$
With these abbreviations in place, we can write the two functionals concisely as \begin{align*}
    L_1(f) = \frac{\poi(f)}{2\pi |f'(1)|(\area(f))^{1/2}} \qquad \mbox{and} \qquad   
    L_2(f) = \frac{\poi(f)}{2\pi |f'(1)|\len(f)}.
\end{align*}
We quickly note variation identities for $\area(\cdot)$ and $\poi(\cdot)$. 
\begin{lemma}\label{first_variation}
    
    Given $f\in \mathbf{F_0}$ and $g$ holomorphic on $\D$ and $C^1$ on $\overline{\D}$, we have \begin{align*}
        \left.\frac{d}{dt}\area(f+tg)\right|_{t=0} &=2 \int_{\D}\re(\overline{f'(w)}g'(w))|dw|^2,\\
    \left.\frac{d}{dt}\poi(f+tg)\right|_{t=0} &=2 \int_{\D}\re(\overline{f'(w)}g'(w))P(w)|dw|^2,\\
    \left.\frac{d}{dt}\left|f'(1)+tg'(1)\right|^2\right|_{t=0} &= 2\re(\overline{f'(1)}g'(1)).
    \end{align*}
\end{lemma}

\begin{proof}
    Note that the functionals $\area(\cdot)$ and $\poi(\cdot)$ as well as $f \mapsto |f'(1)|^2$ can be thought of as  "$L^2$ energies" of $f'$ with respective to measures $|dw|^2$, $P(w)|dw|^2$ and $\delta_1$ respectively. Thus, it is somewhat straightforward to compute their first variations. For instance, 
    we have that \begin{align*}
        \left.\frac{d}{dt}\area(f+tg)\right|_{t=0} &= \left.\frac{d}{dt}\int_\D |f'(w)+tg'(w)|^2|dw|^2\right|_{t=0}\\
        &=\left.\frac{d}{dt}\int_\D \left(2t\re (\overline{f'(w)}g'(w))+t^2|g'(w)|^2\right)|dw|^2\right|_{t=0}\\
        &= 2\int_{\D} \re(\overline{f'(w)}g'(w))|dw|^2.
    \end{align*}  The cases for $\poi(\cdot)$ and $f\mapsto |f'(1)|^2$ follow similarly.
\end{proof}
The functional $\len(\cdot)$ is not an $L^2$ energy for $f'$, we defer the computation of its first variation to Section~\ref{perturbation}. 

\subsection{Structure of the proofs} 
We have converted the optimization problems of $\mathbb{L}_i$ over convex domains to ones of $L_i$ over conformal maps. The purpose of this short section is to explain the different steps that underlie our proof of Theorem~\ref{theo-1} and Theorem~\ref{theo}.

\textit{Proof of Theorem~\ref{theo-1}}. 
In Section~\ref{goingback}, we will present Lemma~\ref{corners} which shows that if $f(\partial \D)$ has a smooth  corner at $f(e^{i\theta_0})$ with angle $\pi\alpha$, then we have \begin{equation*}
    |f'(z)| \asymp |z-e^{i\theta_0}|^{\alpha-1}.
\end{equation*} 
Inspired by this fact, we aim to show in Section~\ref{dom_res} that the optimizer $f$ of $L_1$ satisfies that $f'\in L^2_{\mathrm{loc}}(\partial\D \backslash\set{1})$ (Proposition~\ref{ineq}) and thus $f(\partial \D)$ cannot have smooth acute corners away from $f(1)$. To show this, 
\begin{enumerate}[({1.}1)]
    \item we will first construct an increasing family of disks $\set{D_a}_{a\in (0,1]}$ in $\D$ (see Figure~\ref{fig:circl}). These disks correspond to the level sets of the Poisson kernel $P$.
    \item Using domain restriction (Lemma~\ref{circles}), we will then show that for a conformal map $f:\D\to \Omega$ onto a convex domain, $\set{f|_{D_a}}$ corresponds to an increasing family of convex subsets in $\Omega$.
    \item We then compute the change of $L_1$ over $\set{f|_{D_a}}$ for the maximizing conformal map $f$ of $L_1$ (Proposition~\ref{ineq}). Because of the relation between $\set{D_a}$ and $P$, the change of $\poi(f)$ can be computed explicitly along with that of $\area(f)$.
    \item Since $f$ is the maximizer of $L_1$, the value for $L_1$ should increase as $\set{D_a}$ increase. This would imply the local square integrability of $f'$ (Proposition~\ref{ineq}).
\end{enumerate}

\textit{Proof of Theorem~\ref{theo}}. We will show that for the optimizer $f$ of either $L_1$ or $L_2$, $f(\partial \D)$ has many "almost zero-curvature" points around $f(1)$ (Proposition~\ref{theo3}) in Section~\ref{curv_sign}. To show this,
\begin{enumerate}[({2.}1)] 
    \item  we will define \textit{conformal curvature}, which generalizes the notion of curvature to non-$C^{2+\alpha}$ domains (Proposition~\ref{coinci}).
    \item We will develop a way to apply an Euler-Lagrange argument to the optimizer $f$ with a noise function $g$. We will show how to ensure the nonnegativity of the curvature signature when perturbing $f$. In particular, we will show that by setting the noise function to be $g=\int_{[0,z]}f'h$, we can relate the change of the curvature signature to the outward pointing normal vector of the domain $h(\D)$ (Proposition~\ref{perbu}).
    \item We will then construct a specific class $\set{h_b}$ (Lemma~\ref{nice_fam}). We will prove that if the curvature signature of $f(\D)$ were positive near $f(0)$, then we could use $\set{h_b}$ to perturb $f$.
    \item We will then prove that if we could use $\set{h_b}$ to perturb $f$, then the Euler-Lagrange argument applied to $L_i(f)$ would imply that $f$ is not the maximizer. Thus we can deduce that the curvature signature $k_f$ will have zero points accumulated at $0$ (Proposition~\ref{main}).
    \item This will imply that $f(\partial \D)$ has many zero conformal curvature points around $f(1)$ (Lemma~\ref{rela}).
\end{enumerate}

\textit{Translating the conformal functional $L_i$ to $\mathbb{L}_i$}. So far, we have only worked with the optimal conformal map for $L_i$ instead of the optimal domain for $\mathbb{L}_i$. In Section~\ref{goingback}, 
\begin{enumerate}[({3.}1)] 
    \item we will show that $L_i$ and $\mathbb{L}_i$ attain the same supremum (Theorem~\ref{equiv}). 
    \item Using that, we will then show that if $f$ maximizes $L_i$ and is $C^{1+\alpha}$ near $1$, then $f(\D)$ is in fact the optimal domain for $\mathbb{L}_i$ with $f(1)$ being an extremal point (Section~\ref{proof1}).
    \item  Finally, we will show that if $f$ is $C^{2+\alpha}$ near $1$, then the points with zero conformal curvature on $f(\partial\D)$ near $1$ are exactly the ones with zero Euclidean curvature, proving Theorem~\ref{theo} (Section~\ref{proof2}).
\end{enumerate}

\section{Domain Restriction}\label{dom_res}
\subsection{Analytic Domains}
For a domain $\Omega$, we denote the set of analytic functions on $\Omega$ by $\hol(\Omega)$ and let $\hol(\overline{\Omega})$ denote the set of functions in $\hol(\Omega)$ that can be extended analytically to an open set containing $\overline{\Omega}$. If one (therefore all) Riemann maps of $\Omega$ are in $\hol(\overline{\D})$, then $\partial\Omega$ is a real analytic loop (thus $\Omega$ is a real analytic domain).
Recall that the family of all Riemann maps of bounded convex domain can be characterized as \begin{equation}
    \mathbf{F}:= \set{f \in \hol(\D)|~ f'(z) \neq 0 ~\text{and}~ v_f(z) >0~\text{on $\D$}, f'\in L^2(\D)}.
\end{equation} 
We also recall that for $L_i$ to be well-defined, we would like to consider the following subfamily of $\mathbf{F}$:\begin{equation*}
\mathbf{F_0} = \set{f\in \mathbf{F}: \poi(f) < \infty,~f~\text{has nontangential limit at}~1}.
\end{equation*} Note that Lemma~\ref{increase} implies $f'(1)\neq 0$ whenever it is defined. Then $\mathbf{F_0}$ is the family of all functions in $\mathbf{F}$ where $L_1$ and $L_2$ are well-defined in $(0,\infty)$. Using domain restriction with Lemma~\ref{coinci}, we can approximate $L_i(f)$ using functions in $\hol(\overline{\D})$:

\begin{lemma}\label{approx}
    For any $f\in \mathbf{F_0}$, there exists a sequence of functions $\set{f_n}\subseteq \mathbf{F}_0\cap \hol(\overline{\D})$ such that \begin{enumerate}
        \item The sequence $\set{f_n}$ converges to $f$  normally (uniformly on compact sets).
        \item The sequence $\set{f_n(\D)}$ is an increasing family of domains such that \begin{equation*}
            \bigcup^\infty_{n=1}f_n(\D) = f(\D).
        \end{equation*}
        \item The perimeters of $\set{f_n(\D)}$ also increase to that of $f(\D)$.
        \item Let $u_n$ be the torsion function for domain $\Omega_n := f_n(\D)$. Then \begin{align*}
            \mathbb{L}_1(\Omega_n)\geq &\frac{\left|\frac{\partial u_n}{\partial \mathbf{n}}(f_n(1))\right|}{|\Omega_n|^{1/2}} =L_1(f_n)\to L_1(f),\\
           \mathbb{L}_2(\Omega_n)\geq &\frac{|\frac{\partial u_n}{\partial \mathbf{n}}(f_n(1))|}{\mathcal{H}^1(\partial \Omega_n)} =L_2(f_n)\to L_2(f).
        \end{align*} 
We note that if $f_n(1)$ is an extremal point of $\Omega_n$, then the previous inequalities are equations. 
    \end{enumerate}
\end{lemma}

\begin{proof}
    Given $r\in (0,1]$, we define $f_r(z) = f(rz)$. Then Lemma~\ref{convex} implies $f_r\in \mathbf{F} \cap \hol(\overline{\D})$ for all $r$ and Lemma~\ref{increase} implies that 
     for all $ z\in\D,\quad |f'_r(z)|^2 = r^2|f'(rz)|^2 $ increases to $|f'(z)|^2$ as $r\to 1$.
 Note that $\poi(f)$, $\area(f)$ and $\len(f)$ are integrals of $|f'|^2$. Thus by the monotone convergence theorem, we have 
 $
        \poi(f_r) \uparrow \poi(f)$ as well as $\area(f_r)\uparrow \area(f)$ and $\len(f_r)\uparrow \len(f)$,
as $r$ tends to $1$. Since for $r \in (0,1)$, $f'(r)$ tends to the nontangential limit $f'(1)$, we also have that $ |f'_r(1)| \uparrow |f'(1)|$. Therefore, we have $L_1(f_r) \to L_1(f)$ and $L_2(f_r)\to L_2(f)$.
    However, for $r<1$, $f_r$ is analytic on $\overline{\D}$ (in particular it is $C^\infty(\overline{\D})$). The proof of Lemma~\ref{function} implies that \begin{equation*}
        \left|\frac{\partial u_r}{\partial \mathbf{n}}(f_r(1))\right| =\frac{\int_\D |f_r'(w)|^2P(w)|dw|^2}{2\pi|f'_r(1)|} = \frac{\poi(f_r)}{2\pi|f'_r(1)|},
    \end{equation*}
    where $u_r$ is the torsion function for $f(r\D)$. Thus the result follows by taking the sequence with $r_n = 1-1/n$.
\end{proof}

We know from Proposition~\ref{verysmooth} and Part~(4) of Lemma~\ref{approx} that 
\begin{corollary}\label{hol}
    We have \begin{align*}
        \sup_{f\in\mathbf{F}_0} L_1(f)&=  \sup_{\emph{\tiny $\Omega$~convex, bounded, analytic}}\mathbb{L}_1(\Omega) \leq \frac{1}{\sqrt{2\pi}},\\
        \sup_{f\in\mathbf{F}_0} L_2(f) &= \sup_{\emph{\tiny $\Omega$~convex, bounded, analytic}}\mathbb{L}_2(\Omega) \leq \frac{1}{\sqrt{8}\pi}.
    \end{align*}
\end{corollary}
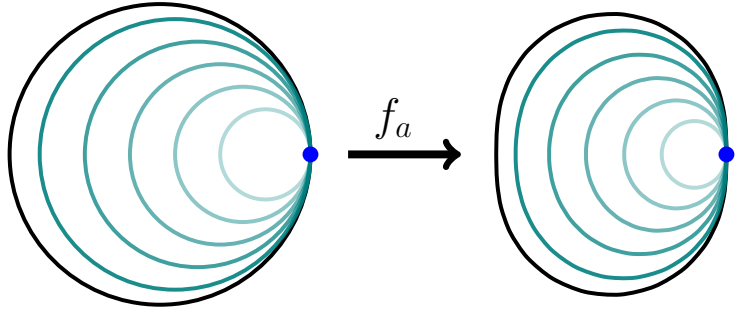
\begin{figure}[ht]
    \centering
    \begin{tikzpicture}\draw [line width=1.5] (-6, 0) circle (2) ;
\draw[domain=0:360,smooth,samples=30,variable=\t, line width=1.5, draw opacity=1]
  plot({1.7*(cos(\t)-0.1*cos(3*\t))},{1.7*(sin(\t)-0.1*sin(3*\t))});
  \cf{0.3}
    \cf{0.45}
    \cf{0.6}
    \cf{0.75}
    \cf{0.9}
    \fill [blue] (-4,0) circle (3pt);
    \fill [blue] ({1.7*0.9},0) circle (3pt);
    \draw [->,line width=3pt] (-3.5,0) -- (-2,0) node[pos=0.4,above=1pt] {\huge $f_a$};
    \end{tikzpicture}
    \caption{Increasing family of disks $\set{D_a}$ and corresponding family of convex domains $\set{f(D_a)}$}
    \label{fig:circl}
\end{figure}

\subsection{Family of Shrinking Disks}

Now let us consider the following family of disks inside $\D$: \begin{equation*}
    D_a = B_a(1-a), \quad a\in(0,1].
\end{equation*}
As we will see soon, these are the level sets of Poisson kernel with $\partial D_a = P^{-1}(\frac{1-a}{a})$. This is the property of this disk family that we want to leverage. For any $f\in\mathbf{F}_0$, $\set{D_a}$ give us a family of conformal maps: \begin{align}\label{circle_fam}
    f_a : &\D \to f(D_a)\\
    &z\mapsto f(az+(1-a)).\nonumber
\end{align} For each $a\in (0,1)$, $f_a$ is a conformal map of bounded convex domain by Lemma~\ref{circles}. It follows that for each $f\in\mathbf{F}_0$, $\set{f_a}$ are a family of maps in $\mathbf{F}$ with $f_1=f$. Additionally, each $f'_a(z) = af'(az+(1-a))$ has a nontangential limit $af'(1)$ at $1$. We want to analyze how $L_1(f_a)$ changes along $a$. It is clear that $|f'_a(1)| = a |f'(1)|$ and thus $\frac{d}{da}|f'_a(1)| = |f'(1)|$. 
For $\area(\cdot)$ and $\poi(\cdot)$, we have \begin{align*}
    \area(f_a) &=\int_\D |f'_a(w)|^2 |dw|^2 =  \int_\D a^2|f'(aw+(1-a))|^2 |dw|^2 = \int_{D_a} |f'(w)|^2 |dw|^2
    \end{align*}
    and
    \begin{align*}
    \poi(f_a) &=\int_\D |f'_a(w)|^2 P(w)|dw|^2 =  \int_\D a^2|f'(aw+(1-a))|^2 P(w)|dw|^2\\
    &= \int_{D_a} |f'(w)|^2P\left(\frac{w-(1-a)}{a}\right) |dw|^2.
\end{align*}

The following lemma gives the formula for the derivatives of $\area(f_a)$ and $\poi(f_a)$.

\begin{lemma}\label{der}
For $a\in (0,1)$, we have that \begin{equation*}
    \frac{d}{da} \area(f_a) = \int_{\partial D_a} |f'(\zeta)|^2\re\left(1-\frac{\zeta}{|\zeta|}\right) |d\zeta|
\end{equation*} and 

\begin{equation*}
    \liminf_{h\downarrow 0} \frac{\area(f_1) - \area(f_{1-h})}{h} \geq \int_{\partial\D} |f'(\zeta)|^2\re(1-\zeta)|d\zeta|.
\end{equation*}
For $a\in(0,1]$, we also have that \begin{equation*}
    \frac{d}{da} \poi(f_a) = \int_{D_a} |f'(w)|^2\left(P(w)+1\right) |dw|^2 >0,
\end{equation*} where we use the left derivative for $a=1$. In particular, $\poi(f_a)$ is increasing in $a$. Therefore, if $f\in \mathbf{F_0}$, the the family $\set{f_a}_{a\in(0,1]}$ is also in $\mathbf{F_0}$. 
\end{lemma}

\begin{proof}
Let us first look at $\area(f_a)$. Using the parametrization $w = re^{i\theta}+(1-r)$ where $r\in (0,1)$ and $\theta \in (-\pi,\pi)$, we have 
\begin{align*}
    \area(f_a) &= \int_{D_a} |f'(w)|^2|dw|^2 = \int_0^a \int_{-\pi}^\pi |f'(re^{i\theta}+(1-r))|^2(1-\cos \theta)rd\theta dr.
\end{align*}
For $a\in (0,1)$, we have $|f'|^2 \in C^{\infty}(\overline{D_a})$. It follows that \begin{align*}
    \frac{d}{da} \area(f_a) &= \int_{-\pi}^\pi |f'(ae^{i\theta}+(1-a))|^2(1-\cos \theta)ad\theta \\
    &= \int_{\partial D_a} |f'(w)|^2\re\left(1 - \frac{\zeta}{|\zeta|}\right)|d\zeta|.
\end{align*}
Note that $\area(f_a)$ is increasing, therefore if we set \begin{equation*}
    X = \liminf_{h\downarrow 0} \frac{\area(f_1) - \area(f_{1-h})}{h}, 
\end{equation*} we have \begin{align*}
    X= &\liminf_{h\downarrow0}\frac{ \int_{1-h}^1 \int_{-\pi}^\pi |f'(re^{i\theta}+(1-r))|^2(1-\cos \theta)rd\theta dr}{h}\\
    = &\liminf_{h\downarrow 0}\int_{-\pi}^\pi \frac{\int_{1-h}^1 |f'(re^{i\theta}+(1-r))|^2rdr}{h}(1-\cos \theta)d\theta,
\end{align*}
where the limit can be infinite. Since $f'$ has nontangential limits almost everywhere on $\partial\D$, it follows that \begin{equation*}
    \lim_{h\downarrow 0} \frac{\int_{1-h}^1 |f'(re^{i\theta}+(1-r))|^2rdr}{h} = |f'(e^{i\theta})|^2,
\end{equation*} for almost all $\theta\in (-\pi,\pi)$. Thus it follows from Fatou's lemma that \begin{align*}
    X \geq \int_{-\pi}^\pi |f'(e^{i\theta})|^2(1-\cos \theta)d\theta 
    = \int_{\partial\D} |f'(\zeta)|^2\re(1-\zeta)|d\zeta|.
\end{align*}

Now for $\poi(f_a)$, we note that for $0<h\leq 1-a$, \begin{align*}
    \poi(f_{a+h}) - \poi(f_a) &= \int_{D_{a+h}} |f'(w)|^2P\left(\frac{w-(1-a-h)}{a+h}\right) |dw|^2 \\
    - \int_{D_a}& |f'(w)|^2P\left(\frac{w-(1-a)}{a}\right) |dw|^2\\
    &=\int_{D_{a+h}\backslash D_a} |f'(w)|^2P\left(\frac{w-(1-a-h)}{a+h}\right) |dw|^2 \\
    + \int_{D_a}& |f'(w)|^2\left(P\left(\frac{w-(1-a-h)}{a+h}\right)-P\left(\frac{w-(1-a)}{a}\right)\right) |dw|^2.
\end{align*} 

A short computation shows that
$$ P\left(\frac{w-(1-a-h)}{a+h}\right)-P\left(\frac{w-(1-a)}{a}\right) = h(P(w)+1).$$

On the other hand, by change of variables, the integral
\begin{align*}
    Y = \int_{D_{a+h}\backslash D_a} |f'(w)|^2P\left(\frac{w-(1-a-h)}{a+h}\right) |dw|^2
    \end{align*}
    can be written as
    $$ Y = (a+h)^2 \int_{\D \backslash D_{\frac{a}{a+h}}}|f'((a+h)w+(1-a-h))|^2 P(w)|dw|^2.$$
  Note that $\partial D_a$ are level sets of $P(z)$ and $P \equiv (1-a)/a$ on $\partial D_a$ and, as a consequence, $P(w) \leq (1-a)/a$ for $w\in \D\backslash D_a$. 
Therefore, \begin{align*}
    Y &\leq (a+h)^2 \int_{\D \backslash D_{\frac{a}{a+h}}}|f'((a+h)w+(1-a-h))|^2 \frac{1-\frac{a}{a+h}}{\frac{a}{a+h}}|dw|^2\\
    =&\frac{(a+h)^2h}{a} \int_{\D \backslash D_{\frac{a}{a+h}}}|f'((a+h)w+(1-a-h))|^2 |dw|^2\\
    =&\frac{h}{a} \int_{D_{a+h}\backslash D_a}|f'(w)|^2|dw|^2 =\frac{h}{a} \int_{\D}|f'(w)|^2 \mathbf{1}_{D_{a+h}\backslash D_a}(w)|dw|^2,
\end{align*} which is $o(h)$ by the dominated convergence theorem. Therefore, it follows that $\frac{d}{da} \poi(f_a) = \int_{D_a} |f'(w)|^2\left(P(w)+1\right) |dw|^2$ for all $a\in (0,1]$. 
\end{proof}

\subsection{Local Square Integrability on $\partial\D$}

In this section, we will show using domain restriction that extremizers cannot have a very large derivative on $\partial \D$.

\begin{proposition}~\label{ineq}
    Suppose $f$ is a conformal map in $\mathbf{F}_0$ that maximizes $L_1$. Then we have that \begin{equation*}
        2(\area(f))^2 \geq \poi(f)\int_{\partial \D}|f'(w)|^2\re(1-w)|dw|.
    \end{equation*} In particular, we have $f'\in L^2_{\mathrm{loc}}(\partial \D \backslash \set{1})$.
\end{proposition}
\begin{proof}
    Suppose $f$ maximizes $L_1$ in $\mathbf{F}_a$. Then for any $a\in(0,1)$ we have $L_1(f_a) \leq L_1(f)$. It follows that 
\begin{equation*}
    \limsup_{a\uparrow 1} \frac{\log L_1(f_1) - \log L_1(f_a)}{1-a} \geq 0.
\end{equation*} By Lemma~\ref{der}, we have 
\begin{align*}
    \limsup_{a\uparrow 1} \frac{\log L_1(f_1) - \log L_1(f_a)}{1-a} = &\frac{\area(f)}{\poi(f)} - \frac{\liminf_{h\downarrow 0} \frac{\area(f_1) - \area(f_{1-h})}{h}}{2\area(f)} \\
    \leq &\frac{\area(f)}{\poi(f)} - \frac{\int_{\partial\D}|f'(w)|^2\re(1-w)|dw|}{2\area(f)}
\end{align*}
Hence, \begin{equation*}
    \frac{\area(f)}{\poi(f)} - \frac{\int_{\partial \D}|f'(w)|^2\re(1-w)|dw|}{2\area(f)}  \geq 0.
\end{equation*} The result then follows.
\end{proof}

\section{Curvature Signature}\label{curv_sign}
\subsection{Conformal Curvature} 
The purpose of this section is to recall some facts about the curvature of the domain for the convenience of the reader and to introduce a new geometric notion that we call \textit{conformal curvature}. We are not aware of this notion existing in the literature, however, it is sufficiently natural that it may well appear somewhere. We give a self-contained exposition of the relevant properties. Let $\Omega$ be a Jordan domain and $f:\D\to \Omega$ be one of its Riemann maps. Recall that we set
\begin{align*}
    v_f(z) = \re\left(1+z\frac{f''(z)}{f'(z)}\right) \qquad \mbox{and} \qquad
    k_f(\theta) = \liminf_{\D\ni z\to e^{i\theta}} v_f(z).
\end{align*} The function $v_f:\D\to \R$ is a harmonic function for locally univalent $f$ and curvature signature $k_f$ is lower semicontinuous. If $f\in C^2(\overline{\D})$, then $t \mapsto f(e^{it})$ is a proper parametrization of the $C^2$ curve $\partial\Omega$. Note that for a properly parametrized curve $\gamma(t) = (x(t),y(t))$ the (signed) Euclidean curvature is given by \begin{equation*}
 \kappa(t) =  \frac{x'(t)y''(t)-y'(t)x''(t)}{|x'(t)^2+y'(t)^2|^{3/2}}.
\end{equation*}
In our case, we have $x(t) = \re f(e^{it})$ and $y(t) = \im f(e^{it})$. By plugging them in, we have that the Euclidean curvature of $\partial\Omega$ at $f(e^{i\theta})$ is given by
\begin{equation}\label{curvatur}
        \frac{1}{|f'(e^{i\theta})|}\re\left(1+e^{i\theta}\frac{f''(e^{i\theta})}{f'(e^{i\theta})}\right) = \frac{v_f(e^{i\theta})}{|f'(e^{i\theta})|}.
\end{equation}
For a $C^{2+\alpha}$ domain, we have $f\in C^{2+\alpha}(\partial\D)$ for any Riemann map $f$ and thus the curvature on the boundary can be calculated using $\eqref{curvatur}$. For non-$C^{2+\alpha}$ domain, we cannot guarantee that $f$ is $C^2(\partial\D)$. Nevertheless, as we will see, it is natural to set the curvature of $f(\partial \D)$ at point $f(e^{i\theta})$ to be \begin{equation}\label{curvature}
    K_{f(\D)}(f(e^{i\theta})) = \liminf_{\D\ni z\to e^{i\theta}} \frac{v_f(z)}{|f'(z)|}.
\end{equation}

We call $K_\Omega(p)$ the \textit{conformal curvature} of point $p\in \partial\Omega$. It follows immediately from the definition that $K_\Omega$ is a lower semicontinuous function on $\partial \Omega$. One relation between the conformal curvature $K_\Omega$ and curvature signature $k_f$ is as follows:

\begin{lemma}\label{domination}
    Let $f$ be a conformal map onto a Jordan domain $\Omega$. Then there exists a constant $C_f\in(0,\infty)$ such that for all $\theta\in [-\pi,\pi]$, \begin{equation}\label{dom}
        K_{\Omega}(f(e^{i\theta})) \leq C_f k_f(\theta).
    \end{equation} 
\end{lemma}
\begin{proof}
     We set \begin{equation*}
         C_f = \frac{2}{\inf_{\theta\in[-\pi,\pi]}\set{|f'(e^{i\theta}/2)|} }\in (0,\infty).
     \end{equation*}
     If $f$ is a Riemann map of a  convex domain, then Lemma~\ref{increase} implies that \begin{equation*}
         |f'(re^{i\theta})| \geq \frac{|f'(e^{i\theta}/2)|}{2r} \geq \frac{C_f^{-1}}{r}\quad \text{for all}\quad r\in (\frac{1}{2},1).
     \end{equation*}
     It follows that \begin{equation*}
         K_\Omega(f(e^{i\theta})) = \liminf_{\D\ni z\to e^{i\theta}} \frac{v_f(z)}{|f'(z)|} \leq C_f\liminf_{\D\ni z\to e^{i\theta}} \frac{v_f(z)}{|z|} =C_fk_f(\theta).
     \end{equation*}
\end{proof}

\label{perturbation}
\subsection{Conformal curvature is well-defined.}\label{well-defined} 
In this section, we will prove that conformal curvature is well-defined, it is an \textit{intrinsic} geometric object that is independent of the choice of the Riemann map $f$.
\begin{proposition}\label{coinci}
    Given any bounded convex domain $\Omega$, the conformal curvature $K_\Omega: \partial \Omega \to [0,\infty]$ is well-defined and independent of the choice of Riemann map $f:\D\to \Omega$. When $\Omega$ is a $C^{2+\alpha}$ domain, $K_\Omega$ coincides with the Euclidean curvature.
\end{proposition}
To prove Proposition~\ref{coinci}, it  suffices to show invariance under automorphisms.
\begin{lemma}\label{rela}
    Given an automorphism $\varphi$ of $\D$ and $g = f\circ \varphi$, we have \begin{equation*}
        \liminf_{\D\ni z\to \varphi^{-1}(e^{i\theta})} \frac{v_f(z)}{|f'(z)|} = \liminf_{\D\ni z\to e^{i\theta}} \frac{v_g(z)}{|g'(z)|}
    \end{equation*} for all $\theta$.
\end{lemma}
\begin{proof}
    Note that $g'(z)= f'(\varphi(z))\varphi'(z)$ and
    \begin{align*}
        g''(z) &= f''(\varphi(z))(\varphi'(z))^2+f'(\varphi(z))\varphi''(z).
    \end{align*} Therefore, we have \begin{align*}
        v_g(z) = \re \left(z\varphi'(z)\frac{f''(\varphi(z))}{f'(\varphi(z))}\right) + \re\left(1+z\frac{\varphi''(z)}{\varphi'(z)}\right).
    \end{align*} Since $\varphi$ is a Riemann map of $\D$ which has constant curvature $1$ on the boundary, we have \begin{align*}
        \lim_{\D\ni z\to e^{i\theta}}\frac{1}{|\varphi'(z)|}\re\left(1+z\frac{\varphi''(z)}{\varphi'(z)}\right)
        =\frac{1}{|\varphi'(e^{i\theta})|}\re\left(1+e^{i\theta}\frac{\varphi''(e^{i\theta})}{\varphi'(e^{i\theta})}\right)=1.
    \end{align*} It follows that \begin{align*}
        \liminf_{\D\ni z\to e^{i\theta}} \frac{v_g(z)}{|g'(z)|} &= \liminf_{\D\ni z\to e^{i\theta}}\frac{\re \left(z\varphi'(z)\frac{f''(\varphi(z))}{f'(\varphi(z))}\right)+|\varphi'(z)|}{|f'(\varphi(z))\varphi'(z)|}\\
        &=\liminf_{\D\ni z\to e^{i\theta}}\frac{1+\re \left(\frac{z\varphi'(z)}{|\varphi'(z)|}\frac{f''(\varphi(z))}{f'(\varphi(z))}\right)}{|f'(\varphi(z))|}.
    \end{align*} However, we have \begin{equation*}
        \lim_{\D \ni z\to e^{i\theta}} \frac{z\varphi'(z)}{\varphi(z)|\varphi'(z)|} =  \frac{e^{i\theta}\varphi'(e^{i\theta})}{\varphi(e^{i\theta})|\varphi'(e^{i\theta})|}= 1.
    \end{equation*} It follows that \begin{align*}
        \liminf_{\D\ni z\to e^{i\theta}} \frac{v_g(z)}{|g'(z)|} &= \liminf_{\D\ni z\to e^{i\theta}}\frac{1+\re \left(\varphi(z)\frac{f''(\varphi(z))}{f'(\varphi(z))}\right)}{|f'(\varphi(z))|}\\
        &= \liminf_{\varphi^{-1}(\D)\ni w\to \varphi^{-1}(e^{i\theta})}\frac{1+\re \left(w\frac{f''(w)}{f'(w)}\right)}{|f'(w)|}\\
        &= \liminf_{\D\ni w\to \varphi^{-1}(e^{i\theta})}\frac{1+\re \left(w\frac{f''(w)}{f'(w)}\right)}{|f'(w)|}.
    \end{align*}
\end{proof}

\subsection{A Euler-Lagrange argument}
In this section, we will show that following:
\begin{proposition}\label{perbu}
    Given $f\in \mathbf{F_0}$ and $h\in \hol(\D)\cap C^1(\overline{\D})$, suppose for all $t>0$ sufficiently small, we have \begin{equation*}
    k_f(\theta) \geq -t\re(e^{i\theta}h'(e^{i\theta}))+2t^2\Vert h\Vert_{L^\infty(\D)}\Vert h'\Vert_{L^\infty(\D)}\quad \text{almost everywhere}.
\end{equation*} Then the following holds:
If $f$ maximizes $L_1$, then \begin{equation}\label{EL2} 2\frac{\int_{\D}|f'(w)|^2\re(h(w))P(w)|dw|^2}{\poi(f)} - \frac{\int_{\D}|f'(w)|^2\re(h(w))|dw|^2}{\area(f)} \leq \re(h(1)).
\end{equation} If $f$ maximizes $L_2$, then \begin{equation}\label{EL3} 2\frac{\int_{\D}|f'(w)|^2\re(h(w))P(w)|dw|^2}{\poi(f)} - \frac{\int_{\partial \D}|f'(w)|\re(h(w))|dw|}{\len(f)} \leq \re(h(1)).
\end{equation}
\end{proposition}

\begin{remark}
     Suppose $f\in\mathbf{F}_0$ is the maximizer of either $L_i$. Using Euler-Lagrange argument for $\log L_i$, we know that for any holomorphic $g: \D\to\C$ such that \begin{equation}\label{perb2}
    f+tg \in \mathbf{F_0} \quad \text{for all $t>0$ sufficiently small,}
\end{equation} we need to have \begin{equation*}
    \left.\frac{d}{dt}\log L_i(f+tg)\right|_{t=0} \leq 0.
\end{equation*} 
\end{remark}

\begin{proof}[Proof of Proposition~\ref{perbu}]
    Let us first look at when Condition~$\eqref{perb2}$ will be true.  Suppose $g$ is a primitive of $f'h$ for some $h\in \hol(\D)\cap C^1(\overline{\D})$ (we write it as $\int f'h$). For $f\in \mathbf{F}$, $|f'|$ is lower-bouned away from $0$ in $\D$. Since $(f+t\int f'h)' = f'(1+th)$, it follows that \begin{equation*}
    \left|\left(f+t\int f'h\right)'(z)\right| > |f'(z)|\left(1-t\Vert h\Vert_{L^\infty(\D)}\right) > 0,
\end{equation*} for all $t>0$ sufficiently small. Moreover, for any $f\in \mathbf{F_0}$, we have \begin{equation*}
    \poi\left(f+t\int f' h\right) \leq \left(1+t \Vert h\Vert_{L^\infty(\D)}\right)^2\poi(f) < \infty.
\end{equation*} Thus to check $f+t\int f' h \in \mathbf{F_0}$, it suffices to check that \begin{align*}
    0< v_{f+t\int f'h}(z)
    &= \re\left(1+z\frac{f''(z)(1+th(z))+tf'(z)h'(z)}{f'(z)(1+th(z))}\right)\\
    &=\re \left(1+z\frac{f''(z)}{f'(z)}\right) + t\re\left(z\frac{h'(z)}{1+th(z)}\right)\\
    &=v_f(z) + t\re\left(zh'(z)\right)-t^2\re\left(\frac{zh(z)h'(z)}{1+th(z)}\right).
\end{align*} Since $h'$ is also bounded on $\D$, we have \begin{equation*}
    \left|\re\left(\frac{zh(z)h'(z)}{1+th(z)}\right)\right| \leq \frac{\Vert h\Vert_{L^\infty(\D)}\Vert h'\Vert_{L^\infty(\D)}} {\left|1-t\Vert h\Vert_{L^\infty(\D)}\right|}.
\end{equation*} Therefore, Condition~\ref{perb2} holds if \begin{equation*}
    v_f(z) > -t \re(zh'(z))+2t^2\Vert h\Vert_{L^\infty(\D)}\Vert h'\Vert_{L^\infty(\D)}\quad \text{for all $t>0$ sufficiently small}.
\end{equation*} Since both sides are harmonic functions and the right-hand side is continuous up to $\partial\D$, the maximum principle implies that Condition~\ref{perb2} is true if for $t>0$ sufficiently small, $k_f(\theta)$ dominate $-t\re\left(e^{i\theta}h'(e^{i\theta})\right)+t^2\Vert h\Vert_{L^\infty(\D)}\Vert h'\Vert_{L^\infty(\D)}$ almost everywhere.
Note that for $g' = f'h$, by Lemma~\ref{first_variation}, we have that 
\begin{align*}
\left.\frac{d}{dt}\log L_1(f+tg)\right|_{t=0} &= 2\frac{\int_{\D}|f'(w)|^2\re(h(w))P(w)|dw|^2}{\poi(f)}\\& - \frac{\int_{\D}|f'(w)|^2\re(h(w))|dw|^2}{\area(f)} - \re(h(1)).
\end{align*}
As for $L_2$, we need to check for the functional $\len(\cdot)$. We have that \begin{align*}
    \len(f+tg) &= \int_{\partial\D} |f'(\zeta)+tg'(\zeta)||d\zeta| =\int_{\partial\D} |f'(\zeta)||1+th(\zeta)||d\zeta|\\
    &= \int_{\partial\D} |f'(\zeta)|\left(1+t\re(h(\zeta))+o(t)\left\Vert h\right\Vert_{L^{\infty}(\partial\D)}\right)|d\zeta|.
\end{align*}

Thus we have \begin{equation*}
    \left.\frac{d}{dt}\len(f+tg)\right|_{t=0} = \int_{\partial\D}|f'(\zeta)|\re(h(\zeta))|d\zeta|.
\end{equation*}
Therefore,

\begin{align*}
\left.\frac{d}{dt}\log L_2(f+tg)\right|_{t=0} &= 2\frac{\int_{\D}|f'(w)|^2\re(h(w))P(w)|dw|^2}{\poi(f)}\\
&- \frac{\int_{\partial\D}|f'(\zeta)|\re(h(\zeta))|d\zeta|}{\len(f)} - \re(h(1)).
\end{align*}

The result then follows.
\end{proof}

\subsection{A nice family of test functions} We now construct a family of test functions $\set{h_b}$ that will be later used in the Euler-Lagrange argument. 
\begin{lemma}\label{nice_fam}
    There exists a family of $C^\infty(\overline{\D})$ conformal maps $\set{h_b}_{b\in(0,1)}$ that satisfies the following properties.
\begin{enumerate}
    \item $|h_b(z)| \leq 10$.
    \item $\re (h_b(1)) < -0.1$.
    \item $\re(h_b(z))\to 0$ pointwise on $\D$ as $b$ tends to $1$.
    \item for each $b$, there exists a closed interval $I(b) \subseteq (0,\pi)$ such that \begin{equation*}
        \set{\theta \in [-\pi,\pi] | \re\left(e^{i\theta}h'_b(e^{i\theta})\right) <0 } \subseteq I(b)
    \end{equation*} and \begin{equation*}
        \inf_{\theta \in [-\pi,\pi]\backslash I(b)}\re\left(e^{i\theta}h'_b(e^{i\theta})\right) >0.
    \end{equation*} 
\end{enumerate}
\end{lemma}

Before proving this, we make the following observation:

\begin{remark}\label{neg}
If $h$ is a conformal map in $C^1(\overline{\D})$, then it is angle-preserving on $\partial\D$ and $h(\partial\D)$ is a $C^1$ curve. Let $\mathbf{n}_\Omega:\partial\Omega \to \partial\D$ be the normal vector of $\Omega$ going outward. Since $h$ is angle-preserving, $\mathbf{n}_{h(\D)}(\zeta)$ is in the direction of velocity vector of $t\mapsto h(te^{i\theta})$ where $\theta = \arg h^{-1}(\zeta)$. Then we have\begin{equation*}
    \re\left(e^{i\theta}h'(e^{i\theta})\right) = |h'(e^{i\theta})|\re\left(\mathbf{n}_{h(\D)}(h(e^{i\theta}))\right) = |h'(e^{i\theta})| \mathbf{n}_{h(\D)}(h(e^{i\theta}))\cdot \begin{pmatrix}
        1\\0
    \end{pmatrix}.
\end{equation*}
This implies that the negativity of $\re\left(e^{i\theta}h'(e^{i\theta})\right)$ for the definition of $I_-(b)$ depends only on outward pointing normal vectors of $h_b(\D)$ around $h(e^{i\theta})$ (note that $h'\neq 0$ by Kellogg's theorem).
\end{remark}

\begin{proof}[Proof of Lemma~\ref{nice_fam}]
We set $h:\D \to B_{1}(-1)$ be the biholomorphism such that $h(-1) = 0$, $h(1) = e^{i\pi/4}-1$ and $\mathrm{Arg}(h(0)+1) = \pi/8$. Note that there are two such functions but both are in $C^\infty(\overline{\D})$. For $b\in [0,1)$, we define \begin{equation*}
    h_b(z) = h\left(\frac{z-b}{1-bz}\right)
\end{equation*} 

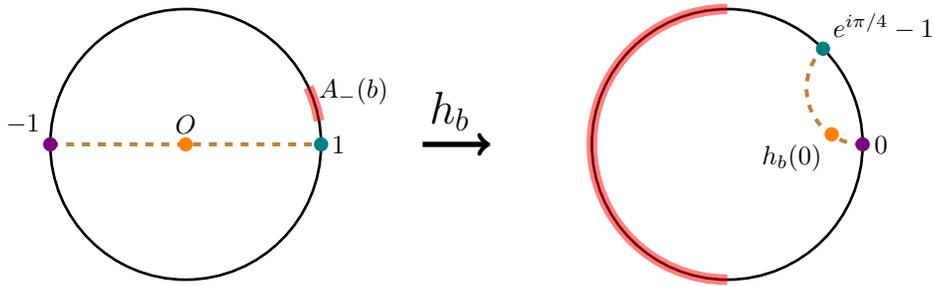
\begin{figure}[ht]
    \centering
    \begin{tikzpicture}[scale=0.9]
        \draw [line width=1pt] (-4,0) circle (2);

        \draw[red, line width=4, opacity=0.5] ({-4+2*cos(10)},{2*sin(10)}) arc(10:25:2);
        
        \draw[brown, dashed, rounded corners=9, line width=1.5] (-6,0) -- (-2,0);
        \fill [violet] (-6,0) circle (3pt) node [black, above left=0.6pt]{$-1$};
        \fill [orange] (-4,0) circle (3pt) node [black, above=0.6pt]{$O$};
        \fill [teal] (-2,0) circle (3pt) node [black, right=0.6pt]{$1$};

        \draw [->,line width=2pt] (-0.5,0) -- (0.5,0) node[pos=0.4,above=1pt] {\huge $h_b$};
    
        \draw [line width=1pt] (4,0) circle (2);

        \draw[red, line width=4, opacity=0.5] ({4+2*cos(90)},{2*sin(90)}) arc(90:270:2);
        
        \harcc{4}{0}{2}{0}{45}{brown, dashed, rounded corners=9, line width=1.5}
        \fill [teal] ({2*cos(45)+4},{2*sin(45)}) circle (3pt) node [black, above right=0.3pt]{$e^{i\pi/4}-1$};
        \fill [orange] ({1.55*cos(5.5)+4},{1.55*sin(5.5)}) circle (3pt) node [black, below left=0.4pt]{$h_b(0)$} ;
        \fill [violet] (6,0) circle (3pt) node [black, right=1pt]{$0$};
        \node at ({-4+2.6*cos(17.5)},{2.6*sin(17.5)}) {$A_-(b)$};

    \end{tikzpicture}
    \caption{Conformal map $h_b$ with circular arcs $A_-(b)$. Note that as $b\to 1$, $h_b(0)\in h((-1,1))$ approaches $0$. Thus the harmonic measure of the left semi-circle tends to $0$ and consequently the arc $A_-(b)$ shrink to $\set{1}$.}
    \label{fig:hb}
\end{figure}

 Then each $h_b$ is still a $C^\infty(\overline{\D})$ map from $\D$ onto $B_1(-1)$ and maps $-1$ and $1$ to $0$ and $e^{i\pi/4}-1$ respectively. We also have $h_b([-1,1]) = h([-1,1])$ (see Figure~\ref{fig:hb}). Moreover, we have \begin{equation*}
    \lim_{b\to 1}h_b(z) = h(-1) = 0,
\end{equation*} for all $z\in\overline{\D} \backslash\set{1}$. Now we define \begin{equation*}
    A_-(b) = h_b^{-1}(\set{e^{i\theta}-1|\theta \in(\pi/2,3\pi/2)}).
\end{equation*}

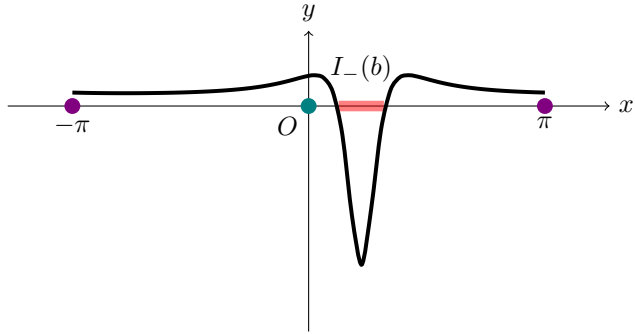
\begin{figure}[ht]
    \centering
    \begin{tikzpicture}
        \draw[->] (-4,0)--(4,0) node[right]{$x$};
        \draw[->] (0,-3)--(0,1) node[above]{$y$};

        \draw[red, line width = 4, opacity=0.5] (0.4,0) -- (1,0);
        \node at (0.7, 0.5) {$I_-(b)$};

        \draw[line width=1.5, domain=-3.141:3.141, smooth, samples=50,variable=\t]
  plot ({\t},{-0.4*(0.7599*cos((\t-0.7) r) - 0.714)/(-1.96*(sin((\t-0.7) r))^2 - 4.172*cos((\t-0.7) r) + 4.1801)+0.1});

        \fill [teal] (0,0) circle (3pt) node [black, below left=0.6pt]{$O$};
        \fill [violet] (-3.141,0) circle (3pt) node [black, below=0.6pt]{$-\pi$};
        \fill [violet] (3.141,0) circle (3pt) node [black, below=0.6pt]{$\pi$};
    \end{tikzpicture}
    \caption{Graph of $\re\left(e^{i\theta}h'_b(e^{i\theta})\right)$. Note that since $A_-(b)$ shrink to $1$ as $b\to 1$, the interval $I_-(b)$ will shrink to $0$.}
    \label{fig:lh}
\end{figure}

Then it follows that $\mathcal{H}^1(A_-(b))/(2\pi)$ is the harmonic measure of the left semi-circle of $\partial B_{1}(-1)$ with respect to $h_b(0)$. We denote the interval on $[-\pi,\pi]$ corresponding to $A_-(b)$ as $I_-(b)$ (see Figure~\ref{fig:lh}). Then we have $|I_-(b)| = \mathcal{H}^1(A_-(b))$ and according to the discussion in Remark~\ref{neg}, \begin{align*}
    \re\left(e^{i\theta}h'_b(e^{i\theta})\right) \begin{cases}
        < 0, & \text{when}\quad \theta \in I_-(b),\\
        = 0, & \text{when}\quad e^{i\theta}\in h_b^{-1}(\set{-\pi/2,\pi/2}),\\
        >0, &\text{otherwise.}
    \end{cases}
\end{align*}

In fact, $h_b$ is a translation of Möbius transformation so one can compute this function explicitly. Using the fact that $h_b(0) \to 0$ (and thus the harmonic measure on $\partial B_{1}(-1)$ with respect to $h_b(0)$ converges to $\delta_{-i}$), we can conclude that $I_-(b)\subseteq (0,\pi)$ and as $b$ tends to $1$, we have $\sup I_-(b) \to 0$. One can take a closed interval $I(b)\subseteq (0,\pi)$ slightly larger than $I_-(b)$ such that $\re\left(e^{i\theta}h'_b(e^{i\theta})\right)$ is lower-bounded away from $0$ outside $I(b)$. The result then follows.

\end{proof}

\subsection{A statement for flatness} Using these properties of $h_b$, we can restrict the behavior of the conformal curvature $k_f$ 
around $0$ for a maximizer $f\in \mathbf{F}_0$ to prove

\begin{proposition}\label{theo3}
    Suppose $f$ is a conformal map in $\mathbf{F}_0$ that maximizes either $L_1$ or $L_2$. Then the boundary of the domain $f(\D)$ has infinitely many points with zero conformal curvature accumulated around $f(1)$.
\end{proposition}

This will help us show Theorem~\ref{theo}. Note that by Lemma~\ref{domination}, we can restate the proposition in terms of $k_f(\theta)$.

\begin{proposition}\label{main}
    Suppose $f\in \mathbf{F}_0$ maximizes either $L_i$. Then we can find an increasing negative sequence $\set{n_j}$ and decreasing positive sequence $\set{p_j}$, both approaching $0$, such that \begin{equation*}
        k_f(n_j) = k_f(p_j) = 0.
    \end{equation*} As $k_f(\theta)$ is lower semicontinuous, we also have $k_f(0)=0$.
\end{proposition}

\begin{proof}
    We will only show the existence of the decreasing positive sequence. The existence of the negative sequence can be proven similarly.

Suppose for the contradiction that $k_f(\theta) >0$ on some interval $(0, a)$ with $a>0$. By lower semi-continuity, $k_f(\theta)$ is lower-bounded away from $0$ on all compact subsets of $(0,a)$.
We take the family of test functions $\set{h_b:\D\to B_{-1}(1)}_{b\in(0,1)}$ in Lemma~\ref{nice_fam}. The lemma implies that there exists $b_0$ such that for all $b\in (0,b_0)$, $I(b)$ is a closed interval in $(0, a)$. It follows that $\inf_{\theta\in I(b)}k_f(\theta) >0$. Since $\theta\mapsto \re\left(e^{i\theta}h'_b(e^{i\theta})\right)$ is smooth, for any $t$  such that \begin{equation*}
     0<t < \min\left(\sqrt{\frac{\inf_{\theta\in I(b)}k_f(\theta)}{4\Vert h\Vert_{L^\infty(\D)}\Vert h'\Vert_{L^\infty(\D)}}},\frac{\inf_{\theta\in I(b)}k_f(\theta)}{2\sup_{\theta\in I(b)}|\re\left(e^{i\theta} h'_b(e^{i\theta})\right)|}\right),
\end{equation*} we have \begin{equation*}
    k_f(\theta) > -t\re\left(e^{i\theta} h'_b(e^{i\theta})\right) + 2t^2\Vert h\Vert_{L^\infty(\D)}\Vert h'\Vert_{L^\infty(\D)},
\end{equation*} for all $\theta$ inside $I(b)$. On the other hand, for any \begin{equation*}
    0<t < \frac{\inf_{\theta\in [-\pi,\pi]\backslash I(b)}\re\left(e^{i\theta} h'_b(e^{i\theta}\right)}{4\Vert h\Vert_{L^\infty(\D)}\Vert h'\Vert_{L^\infty(\D)}},
\end{equation*}
we have for all $\theta\in [-\pi,\pi]\backslash I(b)$\begin{equation*}
    k_f(\theta) \geq 0 > -t\re\left(e^{i\theta} h'_b(e^{i\theta})\right) + 2t^2\Vert h\Vert_{L^\infty(\D)}\Vert h'\Vert_{L^\infty(\D)}.
\end{equation*} Therefore, if $f$ maximizes $L_1$, by inequality $\eqref{EL2}$ in Proposition~\ref{perbu}, we have \begin{equation*} 2\frac{\int_{\D}|f'|^2\re(h_b)P}{\poi(f)} - \frac{\int_{\D}|f'|^2\re(h_b)}{\area(f)} \leq \re(h_b(1)) < -0.1,
\end{equation*} for all $b\in (b_0,1)$. However, $\set{\re h_b}$ is a uniformly bounded sequence of functions converging to $0$ on $\overline{\D}\backslash\set{1}$. Thus by dominated convergence theorem, the left-hand hand of the inequality tends to $0$ as $b$ tends to $1$, causing a contradiction.

Similarly, if $f$ maximizes $L_2$, by inequality $\eqref{EL3}$, we have \begin{equation*} 2\frac{\int_{\D}|f'|^2\re(h_b)P}{\poi(f)} - \frac{\int_{\partial \D}|f'|\re(h_b)}{\len(f)} \leq \re(h_b(1)) < -0.1,
\end{equation*} for all $b\in (0,b_0)$. Again by dominated convergence theorem, the left-hand hand of the inequality tends to $0$ as $b$ tends to $1$, causing a contradiction.
\end{proof}

\section{Going Back to Torsion Function}\label{goingback}
\subsection{Normal Derivative}
    Suppose $f:\D\to \Omega$ is a Riemann map of a bounded convex domain $\Omega$. Since $\partial \Omega$ is rectifiable, we know that $f\in \mathbf{H}^1$. In fact, we can conclude something stronger:

    \begin{lemma}[F. and M. Riesz theorem, \cite{riesz1923randwerte}]\label{fm} Let $\Omega$ be a Jordan domain with rectifiable boundary and let $f:\D\to\Omega$ be a Riemman map of $\Omega$. Then for any $\varphi \in L^1(\partial\Omega)$, we have \begin{equation*}
        \int_{\partial\Omega}\varphi = \int_{\partial\D} \varphi(f(\zeta))|f'(\zeta)||d\zeta|.
    \end{equation*}
    \end{lemma}

    Using F. and M. Riesz theorem, we will show the following:

    \begin{proposition}\label{explicit}
        Let $\Omega$ be a bounded convex domain and let $f:\D\to\Omega$ be a Riemann map of $\Omega$. Then for almost every $\theta\in [-\pi,\pi)$ \begin{equation*}
            \frac{\partial u}{\partial \mathbf{n}}(f(\zeta)) = -\frac{\int_\D P(\zeta,w)|f'(w)|^2|dw|^2}{|f'(\zeta)|}>-\infty.
        \end{equation*}
    \end{proposition}

    \begin{proof}
        We set \begin{equation*}
            u_N(f(\zeta)) = -\frac{\int_\D P(\zeta,w)|f'(w)|^2|dw|^2}{|f'(\zeta)|}.
        \end{equation*} We want to show that for all $\varphi\in C^{\infty}(\R^2)$, $u_N$ satisfies the Gauss-Green formula: \begin{equation}\label{weak}
            \int_{\partial\Omega} \varphi u_N =  \int_\Omega (\varphi\Delta u +u\Delta \varphi) = \int_\Omega (-\varphi +u\Delta \varphi).
        \end{equation}
        Then it follows that $u_N = \partial u/\partial \mathbf{n}$ almost everywhere on $\partial\Omega$ by the uniqueness of weak derivatives. Note that Lemma~\ref{fm} allows us to rewrite $\eqref{weak}$ as
    \begin{align*}
        \int_{\partial\D} \varphi (f(\zeta)) u_N(f(\zeta))|f'(\zeta)||d\zeta|  
        =\int_\D -\varphi(f(w))+u(f(w))\Delta \varphi(f(w)) |f'(w)|^2|dw|^2.
    \end{align*}
    
    For $r\in(0,1)$, we set $f_r(z) = f(rz)$ as in the proof of Lemma~\ref{approx} and let $v_r$ be the solution to \begin{align*}
            -\Delta v_r = |f'_r|^2\quad &\text{inside}\quad \D,\\
        v_r = 0\quad &\text{on}\quad \partial \D.
        \end{align*} Note that a unique solution exists in $C^{\infty}(\overline{\D})$ as $|f'_r|^2 \in C^\infty(\overline{\D})$. This implies that \begin{align*}
            v_r(z) &= \int_\D G(z,w)|f_r'(w)|^2|dw|^2,\\
            \frac{\partial v_r}{\partial \mathbf{n}}(e^{i\theta}) &= -\int_\D P(e^{i\theta},w)|f'_r(w)|^2|dw|^2. 
        \end{align*} Recall that for each $z\in \D$, $|f'_r(z)|$ increases to $|f'(z)|$. Since both $G(z,w)$ and $P(\zeta,w)$ are positive, the monotone convergence theorem implies that \begin{equation*}
            v_r(z) \uparrow v(z),\quad \frac{\partial v_r}{\partial \mathbf{n}}(e^{i\theta}) \downarrow v_N(e^{i\theta}),
        \end{equation*} where \begin{equation*}
            v_N(e^{i\theta}) = -\int_\D P(e^{i\theta},w)|f'(w)|^2|dw|^2 = |f'(e^{i\theta})|u_N(f(e^{i\theta})),
        \end{equation*} 
        which a priori could be $-\infty$. Note that $v_r \in C^{\infty}(\overline{\D})$ and $\frac{\partial v_r}{\partial \mathbf{n}} \in C^{\infty}(\partial\D)$. By the Gauss-Green formula, for any $\psi\in C^{\infty}(\overline{\D})$, we have \begin{equation}\label{gg}
            \int_{\partial\D} \psi(\zeta) \frac{\partial v_r}{\partial \mathbf{n}}(\zeta)|d\zeta|  = \int_\D \left(-\psi(w)|f'_r(w)|^2+v_r(w)\Delta \psi(w)\right)|dw|^2.
        \end{equation} The right-hand side of $\eqref{gg}$ is dominated by the integral \begin{equation*}
            \max(\Vert\psi\Vert_{L^{\infty}(\D)},\Vert\Delta \psi\Vert_{L^{\infty}(\D)})\int_\D \left(|f'(w)|^2 +v(w)\right)|dw|^2.
        \end{equation*} Therefore, it converges as $r\to 1$ to \begin{equation*}
            \int_\D \left(-\psi(w)|f'(w)|^2+v(w)\Delta \psi(w)\right)|dw|^2.
        \end{equation*} In particular, for any positive $\psi \in C^\infty(\overline{\D})$, it follows that \begin{align*}
            \int_{\partial\D} \psi(\zeta) v_N(\zeta)|d\zeta| 
            =&\lim_{r\to 1}\int_{\partial\D} \psi(\zeta) \frac{\partial v_r}{\partial \mathbf{n}}(\zeta)|d\zeta|\\
            = & \int_\D \left(-\psi(w)|f'(w)|^2+v(w)\Delta \psi(w)\right)|dw|^2 > -\infty.
        \end{align*} 
        It follows that $v_N \in L^1(\partial\D)$ and thus $|u_N|$ is finite almost everywhere on $\partial\Omega$. Since every $\psi\in C^{\infty}(\overline{\D})$ is lower-bounded, it follows that \begin{equation}\label{eq8}
            \int_{\partial\D} \psi(\zeta) v_N(\zeta)|d\zeta| = \int_\D \left(-\psi(w)|f'(w)|^2+v(w)\Delta \psi(w)\right)|dw|^2,
        \end{equation} holds for all $\psi \in C^{\infty}(\overline{\D})$. Now for $\varphi \in C^\infty(\R^2)$, we know that $\varphi\circ f_r \in C^\infty(\overline{\D})$ for all $r\in(0,1)$. Plugging it to $\eqref{eq8}$, we have \begin{equation*}
            \int_{\partial\D} \varphi(f(r\zeta)) v_N(\zeta)|d\zeta| = \int_\D -\varphi(f(rw))|f'(w)|^2+v(w)r^2|f'(rw)|^2\Delta \varphi(f(rw))|dw|^2.
        \end{equation*} Note that $(\varphi\circ f_r )v_N$ is dominated by $\left\Vert\varphi\right\Vert_{L^\infty(\overline{\D})}|v_N|\in L^1(\partial\D)$ while the integrant on the right-hand side is dominated by \begin{equation*}
            \left(\left\Vert\varphi\right\Vert_{L^\infty(\overline{\D})} + \left\Vert\Delta\varphi\right\Vert_{L^\infty(\overline{\D})}\left\Vert u\right\Vert_{L^\infty(\Omega)}\right)|f'|^2 \in L^1(\D).
        \end{equation*} Applying the dominated convergence theorem to both sides, we have \begin{equation*}
            \int_{\partial\D} \varphi(f(\zeta)) v_N(\zeta)|d\zeta| = \int_\D -\varphi(f(w))|f'(w)|^2+v(w)|f'(w)|^2\Delta \varphi(f(w))|dw|^2,
        \end{equation*}
            which is
        \begin{equation*}
            \int_{\partial\D} \varphi(f(\zeta)) \frac{v_N(\zeta)}{|f'(\zeta)|}|f'(\zeta)||d\zeta|
            = \int_\D \left(-\varphi(f(w))+v(w)\Delta \varphi(f(w))\right)|f'(w)|^2|dw|^2.
        \end{equation*}
        By change of variables, we have
        \begin{equation*}
            \int_{\partial\Omega} \varphi u_N
            = \int_\Omega \left(-\varphi+u\Delta \varphi\right).
        \end{equation*} The result then follows.
    \end{proof}

    We recall that by Corollary~\ref{hol}, we have \begin{align*}
        \sup_{f\in\mathbf{F_0}} L_1(f) &= \sup_{\text{\tiny $\Omega$~convex, bounded, analytic}} \mathbb{L}_1(\Omega)\\
        \sup_{f\in\mathbf{F_0}} L_2(f) &= \sup_{\text{\tiny $\Omega$~convex, bounded, analytic}} \mathbb{L}_2(\Omega).
    \end{align*}
    Proposition~\ref{explicit} gives us the equivalence between $L_i$ and $\mathbb{L}_i$.
    \begin{theorem}\label{equiv}
      We have \begin{align*}
        \sup_{f\in\mathbf{F_0}} L_1(f) &= \sup_{\emph{\tiny $\Omega$~convex, bounded}} \mathbb{L}_1(\Omega)\\
        \sup_{f\in\mathbf{F_0}} L_2(f) &= \sup_{\emph{\tiny $\Omega$~convex, bounded}} \mathbb{L}_2(\Omega).
    \end{align*}
    \end{theorem}
    \begin{proof}
        It is suffice to show that 
        $$\sup_{f\in\mathbf{F_0}} L_i(f) \geq \sup_{\text{\tiny $\Omega$~convex, bounded}}\mathbb{L}_i(\Omega).$$ Suppose that there exists a bounded convex domain $\Omega$ such that \begin{equation*}
            \mathbb{L}_1(\Omega) \geq \sup_{f\in\mathbf{F_0}} L_1(f) +\varepsilon_0,
        \end{equation*} for some $\varepsilon_0>0$. It follows from Proposition~\ref{explicit} that given a Riemann map $f\to \Omega$, there exists $\zeta_0\in\partial\D$ such that $f'(\zeta_0)$ exists and \begin{equation*}
        \infty>-\frac{\partial u}{\partial \mathbf{n}}(f(\zeta_0)) = \frac{\int_\D P(\zeta_0,w)|f'(w)|^2|dw|^2}{|f'(\zeta_0)| }>\left(\sup_{f\in\mathbf{F_0}} L_1(f)+\varepsilon_0/2\right)|\Omega|^{1/2}.
    \end{equation*} We set $f_0(z) = f(\zeta_0 z)$. Then we have \begin{align*}
        \poi(f_0) &= \int_\D P(w)|f'_0(w)|^2|dw|^2=\int_\D P(w/\zeta_0)|f'(w)|^2|dw|^2\\
        &=\int_\D 2\pi P(\zeta_0,w)|f'(w)|^2|dw|^2 = -2\pi|f'(\zeta_0)|\frac{\partial u}{\partial \mathbf{n}} (f(\zeta_0)) <\infty.
    \end{align*}
    It follows that $f_0\in \mathbf{F}_0$. Besides,\begin{equation*}
        L_1(f_0) = \frac{\poi(f_0)}{2\pi|f'(\zeta_0)||\Omega|^{1/2}} = -\frac{\frac{\partial u}{\partial \mathbf{n}}(f(\zeta_0))}{|\Omega|^{1/2}} > \sup_{f\in\mathbf{F_0}} L_1(f)+\varepsilon_0/2.
    \end{equation*}
    This causes a contradiction. The proof for $L_2$ and $\mathbb{L}_2$ follows similarly.
    \end{proof}

We will now show that there is a natural relationship between the maximizers of $L_i$ and $\mathbb{L}_i$ if they exist. Maximizers of $L_i$ can be used to generate a sequence of convex domains for which $\mathbb{L}_i$ has the same limit and, conversely, maximizing convex domains $\Omega$ give rise to Riemann maps for which $L_i$ converges to the same limit.

    \begin{theorem}
        For $i \in \left\{1,2\right\}$, \begin{enumerate}
            \item Suppose $f\in \mathbf{F}_0$ maximizes $L_i$. Then if we set $\Omega_r = f(r\D)$, then $\set{\Omega_r}$ forms a increasing sequence of analytic convex domains such that
        \begin{equation*}
            \mathbb{L}_i(\Omega_r) \to \sup_{\emph{\tiny $\Omega$~convex, bounded}} \mathbb{L}_1(\Omega)\quad \text{as}\quad r\to 1.
        \end{equation*}
            \item  Suppose $\Omega$ is a convex bounded domain that maximizes $\mathbb{L}_i$. Then there exists a sequence of Riemann maps $f_n:\D \to \Omega$ in $\mathbf{F}_0$ such that \begin{equation*}
                L_i(f_n) \to \sup_{f\in\mathbf{F_0}} L_1(f) \quad \text{as}\quad n\to \infty.
            \end{equation*}
        \end{enumerate}
    \end{theorem}
    
    \begin{proof}
        The first part follows immediately from Part~(5) of Lemma~\ref{approx}. For the second part, we pick any point $z_0 \in \Omega$. For any $p\in\partial \Omega$, we set $f_p$ be the Riemann map of $\Omega$ such that \begin{equation*}
            f_p(0) = z_0,\quad f_p(1) = p.
        \end{equation*} Then we know from Proposition~\ref{explicit} that for almost all $p\in\partial \Omega$, we have $f_p\in \mathbf{F}_0$ and \begin{align*}
            \frac{\partial u}{\partial \mathbf{n}}(p) &= -\frac{\int_\D P(1,w)|f_p'(w)|^2|dw|^2}{|f'_p(1)|} = -\frac{\poi(f_p)}{2\pi|f_p'(1)|}.
        \end{align*} Then we can define a sequence of points $\set{p_n}\subset\partial \Omega$ such that the above equality holds and $-\partial u/{\partial \mathbf{n}}(p_n)$ converges to $\left\Vert \partial u/\partial \mathbf{n}\right\Vert_{L^\infty(\partial\Omega)}$.
    \end{proof}
        
\subsection{Local Smoothness Results}
We need a local version of Kellogg's theorem.
\begin{lemma}[{\cite[Theorem 4.1]{garnett2005harmonic}}]\label{local}
    Suppose $f:\D\to \Omega$ is a conformal map onto a Jordan domain $\Omega$ and  $\gamma' = f(\gamma)$ an open subarc of $\partial \Omega$. Let $k\geq 1$ and $0<\alpha <1$. If $\gamma'$ is $C^{k+\alpha}$, then $f$ is $C^{k+\alpha}$ on compact subsets of $\D\cup \gamma$ and $f'\neq 0$ on $\gamma$.
\end{lemma}

We also need the following result for the Poisson equation:
\begin{lemma}[{\cite[Theorem~5.1]{littman1963regular} \& \cite[Lemma 4.2]{gilbarg1977elliptic}}]\label{local2}
    Given $h \in L^1(\D)$, the Poisson equation\begin{align*}
        -\Delta u = h\quad &\text{inside}\quad \D,\\
        u = 0\quad &\text{on}\quad \partial \D.
    \end{align*}
    admits a unique weak solution with $u\in W_0^{1,p}(\D)$ for all $1\leq p<2$. Let $\zeta\in \partial\D$ be any point such that $h$ is Hölder-continuous on a neighborhood $U$ of $\zeta$. Then $\nabla u$ is continuous at $\zeta$ with \begin{equation*}
        |\nabla u(\zeta)|  = \int_{\D} P(\zeta,w)h(w)|dw|^2.
    \end{equation*}
\end{lemma}

To prove Theorem~\ref{theo-1}, we need to extract information from $C^{1+\alpha}$-corners of the domain. For that we have the following:

\begin{lemma}[{\cite[Theorem 3.9]{pommerenke2013boundary}}]\label{corners}
    If $\partial\Omega$ has two Dini-smooth subarcs (that is, curves parametrized by functions with Dini-continuous derivative), which meet at $\zeta\in\partial \Omega$ with an angle of $\pi\alpha$ for $0<\alpha\leq 1$, then the function $|f'(z)|/|z-\zeta|^{\alpha-1}$ is continuous and bounded away from $0$ around $\zeta$.
\end{lemma}

Note that Hölder-smooth open arcs are Dini-smooth as well.

\subsection{Proof of Theorem~\ref{theo-1}} \label{proof1}
Suppose $u$ be the torsion function of $\Omega$. Let $f:\D\to \Omega$ be the Riemann map of $\Omega$ such that $f(1)$ is an extremal point contained in some $C^{1+\alpha}$-smooth subarc $\gamma$ of $\partial\Omega$. We also set $v=u\circ f$, which satisfies $\eqref{torsion}$. By Lemma~\ref{local}, $f'$ is $\alpha$-Hölder-continuous near $1$. Then Lemma~\ref{local2} implies that \begin{align*}
    L_1(f) &= \frac{\int_\D |f'(w)|^2P(w)|dw|}{2\pi |f'(1)|}|\Omega|^{-1/2} \\&= \frac{\int_\D P(1,w)|f'(w)|^2|dw|}{|f'(1)|}|\Omega|^{-1/2}
    = \frac{|\nabla (u\circ f)(1)|}{|f'(1)|}|\Omega|^{-1/2}.
\end{align*}

Note that Lemma~\ref{local2} also tells us that $\nabla (u\circ f)(z) = f'(z)\cdot \nabla u(f(z))$ is continuous at $1$. Since $f'$ is continuous at $1$ and $f'(1)\neq 0$. It follows that $\nabla u(f(z))$ is continuous at $1$ and  \begin{equation*}
    \nabla u(f(1)) = f'(1)^{-1}\cdot \nabla(u\circ f)(1).
\end{equation*} Therefore, \begin{align*}
    L_1(f) = \frac{|\nabla(u\circ f)(1)|}{|f'(1)||\Omega|^{1/2}} = \frac{|\nabla u(f(1))|}{|\Omega|^{1/2}} = \frac{\left|\frac{\partial u}{\partial\mathbf{n}}(f(1))\right|}{|\Omega|^{1/2}}= \frac{\left\Vert\frac{\partial u}{\partial\mathbf{n}}\right\Vert_{L^{\infty}(\partial\Omega)}}{|\Omega|^{1/2}} = \mathbb{L}_1(\Omega).
\end{align*}

It follows from Theorem~\ref{equiv} that $f$ maximizes $L_1$. By Proposition~\ref{ineq}, we have $f' \in L^2_{\mathrm{loc}}(\partial \D\backslash\set{1})$. Note that $|z-\zeta|^{\alpha-1} \notin L^2_{\mathrm{loc}}(\partial\D \backslash\set{1})$ for any $\zeta\in \partial\D \backslash\set{1}$ and $\alpha < 1/2$. Therefore, once we have the $L^2$-integrability of $f'$ away from $1$, the theorem will follow from Lemma~\ref{corners}.

\subsection{Proof of Theorem~\ref{theo}}\label{proof2}
Given the Riemann map $f: \D \to \Omega$ of $\Omega$, we might assume that $f(1)$ is an extremal point contained in a $C^{2+\alpha}$ subarc. Then $f$ is $C^{2+\alpha}$ near $1$. Using the same argument as above, we have that $f \in \mathbf{F_0}$ and \begin{equation*}
     L_i(f) = \mathbb{L}_i(\Omega) = m_i.
\end{equation*}
Thus $f$ maximizes $L_i$. In this case, $t\to f(e^{it})$ is a $C^2$ parametrization of $\partial\Omega$ near $f(1)$. Therefore, the conformal curvature we defined in $\eqref{curvature}$ is the actual Euclidean curvature of $\partial\Omega$ near $f(1)$. Thus the result then follows from Proposition~\ref{theo3}.

\section{Numerics}
In this section, we will briefly describe the numerical construction of the two "nearly optimal" domains in Figure~\ref{fig:optimal}. The construction is due to Guido Sweers \cite{hoskins2021towards}. Consider the domain
$$    E_q = \set{(x,y)\in \R^2| x^2+\frac{y^2}{q^2} < 1}$$
and the map $h_q: E_q \rightarrow \mathbb{C}$ given by
  $$h_q = z+ \frac{3q^2+1}{4q^4+5q^2+1}z^2+\frac{1}{3(4q^2+1)}z^3.$$  
The domain $h_q(E_q)$ is convex for $q\in (1,2)$ and $h_q(-1)$ is the unique extremal point. Let $u_q$ be the torsion function of $h_q(E_q)$ and $v_q = u_q\circ h_q$. Then we have \begin{align*}
    -\Delta v_q = |h'_q|^2 \quad &\text{inside $E_q$},\\
    v_q = 0 \quad &\text{on $\partial E_q$}.
\end{align*} It follows that $v_q$ is a polynomial of degree $6$ with a factor $(x^2+(y/q)^2-1)$. We can solve $v_q$ explicitly. Then \begin{equation*}
    \frac{\partial u_q}{\partial\mathbf{n}}(h_q(-1)) = \frac{1}{|h
    '_q(-1)|} \frac{\partial v_q}{\partial \mathbf{n}}(-1).
\end{equation*}

The area and perimeter of $h_q(E_q)$ are given by \begin{align*}
    |h_q(E_q)|= \int_{E_q} |h'_q(w)|^2|dw|^2 \qquad \mbox{and} \qquad
    \mathcal{H}^1(h_q(\partial E_q))= \int_{\partial E_q} |h'_q(\zeta)||d\zeta|.
\end{align*}

Thus we can evaluate $L_1(h_q(E_q))$ and $L_2(h_q(E_q))$ numerically. One finds that $L_1(h_q(E_q))$ is maximized for $q\approx 1.3866$ and $L_2(h_q(E_q))$ is maximized for $q\approx 1.2278$. This gives us the two domains in Figure~\ref{fig:optimal}.

\bibliographystyle{plain}

\end{document}